\documentclass[letterpaper,12pt,titlepage,twoside,openright,final]{report}

\newcommand{\href}[1]{#1} % does nothing, but defines the command so the
\usepackage{ifthen}
\newboolean{PrintVersion}
\setboolean{PrintVersion}{false}
\usepackage{amsmath,amssymb,amstext,amsthm} % Lots of math symbols and environments
\usepackage[pdftex]{graphicx} % For including graphics N.B. pdftex graphics driver
\usepackage{enumerate}
\usepackage{caption}
\usepackage{blindtext}

\usepackage{tikz}
\usepackage{subfig}

\usepackage{verbatim}

\usepackage{makeidx}
\makeindex

\swapnumbers

\newtheorem{theorem}{Theorem}[section]
\newtheorem{lemma}[theorem]{Lemma}
\newtheorem{corollary}[theorem]{Corollary}
\theoremstyle{definition}

\newtheorem{remark}{Remark}[section]
\newtheorem*{proofthm}{Proof of Theorem 4.5.1}
\newtheorem*{example}{Example}
\newtheorem{conj}[theorem]{Conjecture}

\newtheorem*{thm1}{Theorem}
\newtheorem*{thm2}{Theorem}
\newtheorem*{thm3}{Theorem}

\usepackage[pdftex,letterpaper=true,pagebackref=false]{hyperref} % with basic options
\hypersetup{
    plainpages=false,       % needed if Roman numbers in frontpages
    pdfpagelabels=true,     % adds page number as label in Acrobat's page count
    bookmarks=true,         % show bookmarks bar?
    unicode=false,          % non-Latin characters in Acrobat’s bookmarks
    pdftoolbar=true,        % show Acrobat’s toolbar?
    pdfmenubar=true,        % show Acrobat’s menu?
    pdffitwindow=false,     % window fit to page when opened
    pdfstartview={FitH},    % fits the width of the page to the window
    pdftitle={uWaterloo\ LaTeX\ Thesis\ Template},    % title: CHANGE THIS TEXT!
    pdfnewwindow=true,      % links in new window
    colorlinks=true,        % false: boxed links; true: colored links
    linkcolor=blue,         % color of internal links
    citecolor=green,        % color of links to bibliography
    filecolor=magenta,      % color of file links
    urlcolor=cyan           % color of external links
}
\ifthenelse{\boolean{PrintVersion}}{   % for improved print quality, change some hyperref options
\hypersetup{	% override some previously defined hyperref options
    citecolor=black,%
    filecolor=black,%
    linkcolor=black,%
    urlcolor=black}
}{} % end of ifthenelse (no else)

\let\origdoublepage\cleardoublepage
\newcommand{\clearemptydoublepage}{%
  \clearpage{\pagestyle{empty}\origdoublepage}}
\let\cleardoublepage\clearemptydoublepage

\newcommand\prooflem{\noindent\textsl{Proof of Lemma 3.2.2. }}

\begin{document}

% For a large document, it is a good idea to divide your thesis
% into several files, each one containing one chapter.
% To illustrate this idea, the "front pages" (i.e., title page,
% declaration, borrowers' page, abstract, acknowledgements,
% dedication, table of contents, list of tables, list of figures,
% nomenclature) are contained within the file "uw-ethesis-frontpgs.tex" which is
% included into the document by the following statement.
%----------------------------------------------------------------------
% FRONT MATERIAL
%----------------------------------------------------------------------
% T I T L E   P A G E
% -------------------
% Last updated May 24, 2011, by Stephen Carr, IST-Client Services
% The title page is counted as page `i' but we need to suppress the
% page number.  We also don't want any headers or footers.
\pagestyle{empty}
\pagenumbering{roman}

% The contents of the title page are specified in the "titlepage"
% environment.
\begin{titlepage}
        \begin{center}
        \vspace*{1.0cm}

        \Huge
        {\bf Entropy and Graphs }

        \vspace*{1.0cm}

        \normalsize
        by \\

        \vspace*{1.0cm}

        \Large
        Seyed Saeed Changiz Rezaei \\

        \vspace*{3.0cm}

        \normalsize
        A thesis \\
        presented to the University of Waterloo \\ 
        in fulfillment of the \\
        thesis requirement for the degree of \\
        Master of Math \\
        in \\
        Combinatorics and Optimization \\

        \vspace*{2.0cm}

        Waterloo, Ontario, Canada, 2013 \\

        \vspace*{1.0cm}

        \copyright\ Seyed Saeed Changiz Rezaei 2013 \\
        \end{center}
\end{titlepage}

% The rest of the front pages should contain no headers and be numbered using Roman numerals starting with `ii'
\pagestyle{plain}
\setcounter{page}{2}

\cleardoublepage % Ends the current page and causes all figures and tables that have so far appeared in the input to be printed.
% In a two-sided printing style, it also makes the next page a right-hand (odd-numbered) page, producing a blank page if necessary.

% D E C L A R A T I O N   P A G E
% -------------------------------
  % The following is the sample Delaration Page as provided by the GSO
  % December 13th, 2006.  It is designed for an electronic thesis.
  \noindent
\begin{center}\textbf{Author's Declaration}\end{center}
I hereby declare that I am the sole author of this thesis. This is a true copy of the thesis, including any required final revisions, as accepted by my examiners.

  \bigskip
  
  \noindent
I understand that my thesis may be made electronically available to the public.

\cleardoublepage
%\newpage

% A B S T R A C T
% ---------------

\begin{center}\textbf{Abstract}\end{center}

The entropy of a graph is a functional depending both on the graph itself and on a probability distribution on its vertex set. This graph functional originated  from the problem of source coding in information theory and was introduced by J. K\"{o}rner in 1973. Although the notion of graph entropy has its roots in information theory, it was proved to be closely related to some classical and frequently studied graph theoretic concepts. For example, it provides an equivalent definition for a graph to be perfect and it can also be applied to obtain lower bounds in graph covering problems. 

In this thesis, we review and investigate three equivalent definitions of graph entropy and its basic properties. Minimum entropy colouring of a graph was proposed by N. Alon in 1996. We study  minimum entropy colouring and its relation to graph entropy. We also discuss the relationship between the entropy and the fractional chromatic number of a graph which was already established in the literature.

A graph $G$ is called \emph{symmetric with respect to a functional $F_G(P)$} defined on the set of all the probability distributions on its vertex set if the distribution $P^*$ maximizing $F_G(P)$ is uniform on $V(G)$. Using the combinatorial definition of the entropy of a graph in terms of its vertex packing polytope and the relationship between the graph entropy and fractional chromatic number, we prove that vertex transitive graphs are symmetric with respect to graph entropy. Furthermore, we show that a bipartite graph is symmetric with respect to graph entropy if and only if it has a perfect matching. As a generalization of this result, we characterize some classes of symmetric perfect graphs with respect to graph entropy. Finally, we prove that the line graph of every bridgeless cubic graph is symmetric with respect to graph entropy. 

\cleardoublepage
%\newpage

% A C K N O W L E D G E M E N T S
% -------------------------------

\begin{center}\textbf{Acknowledgements}\end{center}

I would like to thank my advisor Chris Godsil for his guidance and support throughout my graduate studies in Combinatorics and Optimization Department.  I am also grateful to the Department of Combinatorics and Optimization for providing me with a motivating academic environment.
\cleardoublepage
%\newpage

%% D E D I C A T I O N
%% -------------------
%
%\begin{center}\textbf{Dedication}\end{center}
%
%%This is dedicated to the one I love.
%\cleardoublepage
%%\newpage

% T A B L E   O F   C O N T E N T S
% ---------------------------------
\renewcommand\contentsname{Table of Contents}
\tableofcontents
\cleardoublepage
\phantomsection
%\newpage

%% L I S T   O F   T A B L E S
%% ---------------------------
%\addcontentsline{toc}{chapter}{List of Tables}
%\listoftables
%\cleardoublepage
%\phantomsection		% allows hyperref to link to the correct page
%%\newpage
%
%% L I S T   O F   F I G U R E S
%% -----------------------------
\addcontentsline{toc}{chapter}{List of Figures}
\listoffigures
\cleardoublepage
\phantomsection		% allows hyperref to link to the correct page
%\newpage

% L I S T   O F   S Y M B O L S
% -----------------------------
% To include a Nomenclature section
% \addcontentsline{toc}{chapter}{\textbf{Nomenclature}}
% \renewcommand{\nomname}{Nomenclature}
% \printglossary
% \cleardoublepage
% \phantomsection % allows hyperref to link to the correct page
% \newpage

% Change page numbering back to Arabic numerals
\pagenumbering{arabic}

%----------------------------------------------------------------------
% MAIN BODY
%----------------------------------------------------------------------
% Because this is a short document, and to reduce the number of files
% needed for this template, the chapters are not separate
% documents as suggested above, but you get the idea. If they were
% separate documents, they would each start with the \chapter command, i.e,
% do not contain \documentclass or \begin{document} and \end{document} commands.
%======================================================================
\chapter{Introduction}

The entropy of a graph is an information theoretic functional which is defined on a graph with a probability density on its vertex set. This functional was originally proposed by J. K\"{o}rner in 1973 to study the minimum number of codewords required for representing an information source (see J. K\"{o}rner \cite{JKor}).

J. K\"{o}rner investigated the basic properties of the graph entropy in several papers from 1973 till 1992 (see J. K\"{o}rner \cite{JKor}-\cite{Jkor4}). 

Let $F$ and $G$ be two graphs on the same vertex set $V$. Then the union of graphs $F$ and $G$ is the graph $F\cup G$ with vertex set $V$ and its edge set is the union of the edge set of graph $F$ and the edge set of graph $G$. That is
\begin{eqnarray}
&&V\left(F\cup G\right) = V,\nonumber\\
&&E\left(F\cup G\right) = E\left(F\right)\cup E\left(G\right).\nonumber
\end{eqnarray}
The most important property of the entropy of a graph is that it is sub-additive with respect to the union of graphs. This leads to the application of graph entropy for graph covering problem as well as the problem of perfect hashing. 

The graph covering problem can be described as follows. Given a graph $G$ and a family of graphs $\mathcal G$ where each graph $G_i\in \mathcal G$ has the same vertex set as $G$, we want to cover the edge set of $G$ with the minimum number of graphs from $\mathcal G$. Using the sub-additivity of graph entropy one can obtain lower bounds on this number.

Graph entropy was used in a paper by Fredman and Koml\'{o}s for the minimum number of perfect hash functions of a given range that hash all $k$-element subsets of a set of a given size (see Fredman and Koml\'{o}s \cite{FK}). 

As another application of graph entropy, Kahn and Kim in \cite{Kahn} proposed a sorting algorithm based on the entropy of an appropriate comparability graph.

In 1990, I. Csisz\'{a}r, J. K\"{o}rner, L. Lov\'{a}sz, K. Marton, and G. Simony, characterized minimal pairs of convex corners which generate the probability density $\mathbf p=(p_1,\cdots,p_k)$ in a $k$-dimensional space. Their study led to another definition of the graph entropy in terms of the vertex packing polytope of the graph. They also gave another characterization of a perfect graph using the sub-additivity property of graph entropy. 

The sub-additivity property of the graph entropy was further studied in J. K\"{o}rner \cite{JKor01}, J. K\"{o}rner and G. Longo \cite{JKor11}, J. K\"{o}rner and et.~al. \cite{JKor2}, and J. K\"{o}rner and K. Marton \cite{JKor3}.  Their studies led to the notion of a class of graphs which is called \emph{normal graphs}. 

A set $\mathcal A$ consisting of some subsets of the vertices of a graph $G$ is a \emph{covering}, if every vertex of $G$ is contained in an element of $\mathcal A$. 

A graph $G$ is called a \emph{normal graph}\index{normal graph}, if it admits two coverings $\mathcal C$ and $\mathcal S$ such that every element $C$ of $\mathcal C$ induces a clique and every element $S$ of $\mathcal S$ induces a co-clique, and the intersection of any element of $\mathcal C$ and any element of $\mathcal S$ is nonempty, i.e.,
\[
C\cap S \neq \emptyset,~\forall C\in \mathcal C,~S\in \mathcal S.
\]
It turns out that one can consider normal graphs as a generalization of perfect graphs, since every perfect graph is a normal graph (see J. K\"{o}rner \cite{JKor01} and C. De Simone and J. K\"{o}rner \cite{De}).

Noga Alon, and Alon Orlitsky studied the problem of source coding in information theory using the minimum entropy colouring of the characteristic graph associated with a given information source. They investigated the relationship between the minimum entropy colouring of a graph and the graph entropy (see N. Alon and A. Orlitsky \cite{Alon96}).

This thesis is organized as follows. In Chapter 2, we define the entropy of a random variable. We also briefly investigate the application of entropy in counting problems. In chapter 3, we define the entropy of a graph. Let $VP(G)$ be the \emph{vertex packing polytope} of a given graph $G$ which is the convex hull of the characteristic vectors of its independent sets. Let $|V(G)| = n$ and $P$ be a probability density on $V(G)$. Then the \emph{entropy of $G$ with respect to the probability density $P$} is defined as
\[
H_k(G,P) = \min_{\mathbf{a}\in VP(G)} \sum_{i=1}^n p_i\log (1/a_i).\label{eq:combent}
\]
 This is the definition of graph entropy which we work with throughout this thesis and was given by I. Csisz\'{a}r and et.~al. in \cite{Csis}. However, the origininal denition of graph entropy was given by J. K\"{o}rner \cite{JKor} in the context of source coding problem in information theory and is as follows. Let $G^{(n)}$ be the \emph{$n$-th conormal power} of the given graph $G$ with vertex set $V\left(G^{(n)}\right)= V^n$ and edge set $E^{(n)}$ as
\[
E^{(n)} = \{(x,y) \in V^n \times V^n :  \exists i :
(x_i,y_i)\in E\}.
\]
Furthermore, let 
\[
T_\epsilon^{(n)} = \{U\subseteq V^n: P^n(U)\geq 1 - \epsilon\}.
\]
 Then J. K\"{o}rner \cite{JKor} defined graph entropy $H_k(G,P)$ as
\begin{equation}\label{eq:ent1}
H(G,P) = \lim_{n\rightarrow\infty} \min_{U\in T_\epsilon^{(n)}}  \frac{1}{n} \log\chi(G^{(n)}[U]).
\end{equation}
It is shown in I. Csisz\'{a}r and et. al. in \cite{Csis} that the above two definitions are equal.
We also investigate the basic properties of graph entropy and explain the relationship between the the graph entropy and perfect graphs and fractional chromatic number of a graph. Chapter 4 is devoted to minimum entropy colouring of a given graph and its connection to the graph entropy. G. Simonyi in \cite{Simu} showed that the maximum of the graph entropy of a given graph over the probability density of its vertex set is equal to its fractional chromatic number. We call a graph is symmetric with respect to graph entropy if the uniform density maximizes its entropy. We show that vertex transitive graphs are symmetric. In Chapter 5, we study some other classes of graphs which are symmetric with respect to graph entropy.  Our main results are the following theorems.
\begin{thm1}
Let $G$ be a bipartite graph with parts $A$ and $B$, and no isolated vertices. Then, uniform probability distribution $U$ over the vertices of $G$ maximizes $H_k\left(G,P\right)$ if and only if $G$ has a perfect matching.
\end{thm1}
As a generalization of this result we show that
\begin{thm2}
Let $G=(V,E)$ be a perfect graph and $P$ be a probability distribution on $V(G)$. Then $G$ is symmetric with respect to graph entropy $H_k\left(G,P\right)$ if and only if $G$ can be covered by its cliques of maximum size.
\end{thm2}
A. Schrijver \cite{Schriv1} calls a graph $G$ a \emph{$k$-graph}\index{$k$-graph} if it is $k$-regular and its fractional edge coloring number $\chi_f^\prime(G)$ is equal to $k$. We show that
\begin{thm3}
Let $G$ be a $k$-graph with $k\geq3$. Then the line graph of $G$ is symmetric with respect to graph entropy.
\end{thm3}
As a corollary to this result we show that the line graph of every bridgeless cubic graph is symmetric with respect to graph  entropy. 
\chapter{Entropy and Counting}

In this chapter, we explain some probabilistic preliminaries such as the notions of probability spaces, random variables and the entropy of a random variable. Furthermore, we give some applications of entropy methods in counting problems.

\section{Probability Spaces, Random Variables, and Density functions}

Let $\Omega$ be a set of \emph{outcomes}, let $\mathcal{F}$ be a family of subsets of $\Omega$ which is called the set of \emph{events}, and let  $P:\mathcal{F}\rightarrow[0,1]$ be a function that assigns probabilities to events. The triple $\left(\Omega, \mathcal{F}, P\right)$ is a \emph{probability space}\index{probability!probability space|(}\index{probability!probability space|)} . A \emph{measure}\index{measure!measure} is a nonnegative countably additive set function, that is a function $\mu:\mathcal{F}\rightarrow\mathbb{R}$ such that
\begin{enumerate}[(i).]
\item $\mu(A)\geq\mu(\emptyset)=0$ for all $A\in\mathcal{F}$, and
\item if $A_i\in\mathcal{F}$ is a countable sequence of disjoint sets, then
\[
\mu\left(\bigcup_iA_i\right) = \sum_i\mu\left(A_i\right).
\]
\end{enumerate}

If $\mu\left(\Omega\right)=1$, we call $\mu$ a \emph{probability measure}\index{probability!probability measure|(}. Throughout this thesis, probability measures are denoted by $P(.)$. A probability space is discrete if $\Omega$ is countable. In this thesis, we only consider discrete probability spaces. Then having $p(\omega)\geq0$ for all $\omega\in\Omega$ and $\sum_{\omega\in\Omega}p(\omega) = 1$, for all event $A\in \mathcal{F}$, the probability of the event $A$ is denoted by $P(A)$, which is
\[
P(A) = \sum_{\omega\in A} p(w)
\]
Note that members of $\mathcal{F}$ are called \emph{measurable sets}\index{measure!measurable sets} in measure theory; they are also called \emph{events} in a probability space.

On a finite set $\Omega$, there is a natural probability measure\index{probability!probability measure|)} $P$, called the (discrete) \emph{uniform} measure on $2^{\Omega}$, which assigns probability $\frac{1}{|\Omega|}$ to singleton $\{\omega\}$ for each $\omega$ in $\Omega$. Coin tossing gives us examples with $|\Omega| = 2^n$, $n = 1,2,\cdots$. Another classical example is a fair \emph{die}, a perfect cube which is thrown at random so that each of the six faces, marked with the integers 1 to 6, has equal probability $\frac{1}{6}$ of coming up.

Probability spaces become more interesting when random variables are defined on them. Let $(S,\mathcal S)$ be a measurable space. A function
\[
X:\Omega\rightarrow S,
\]
is called a \emph{measurable map}\index{measure!measurable map}  from $\left(\Omega,\mathcal F\right)$ to $\left(S,\mathcal S\right)$ if
\[
X^{-1} \left(B\right) = \{\omega: X(\omega)\in B\}\in \mathcal{F}.
\]
If $(S,\mathcal S) = (\mathbb R,\mathcal R)$, the real valued function $X$ defined on $\Omega$ is a \emph{random variable}\index{random!random variable|(}.

For a discrete probability space $\Omega$ any function $X:\Omega\rightarrow\mathbb{R}$ is a random variable. The \emph{indicator function} $1_A(\omega)$ of a set $A\in\mathcal{F}$ which is defined as

\begin{equation}
1_A(\omega) = \left\{ \begin{array}{rcl}
1, & & \omega\in A ,\\
0, && \omega\notin A.
\end{array}\right.
\end{equation}
is an example of a random variable.
If $X$ is a random variable, then $X$ induces a probability measure on $\mathbb{R}$ called its \emph{probability density function}\index{probability!probability density function} by setting
\[
\mu(A)=P(X\in A)
\]
for sets $A$. Using the notation introduced above, the right-hand side can be written as $P(X^{-1}(A))$. In words, we pull $A\subseteq\mathbb{R}$ back to $X^{-1}(A)\in \mathcal{F}$ and then take $P$ of that set. For a comprehensive study of probability spaces see R. M. Duddley \cite{prob1} and Rick Durrett \cite{prob2}.

In this thesis, we consider discrete random variables. Let $X$ be a discrete random variable with alphabet $\mathcal{X}$ and probability density function $p_X(x)=\mathrm{Pr}\{X=x\}$, $x\in\mathcal{X}$. For the sake of convenience, we use $p(x)$ instead of $p_X(x)$. Thus, $p(x)$ and $p(y)$ refer to two different random variables and are in fact different probability density functions, $p_X(x)$ and $p_Y(y)$, respectively.

\section{Entropy of a Random Variable}

Let $X$ be a random variable $X$ with probability density $p(x)$. We denote the \emph{expectation} by $E$. Then \emph{expected value of the random variable} $X$ is written
\[
E \left(X\right) = \sum_{x\in\mathcal X}xp(x),
\]
and for a function $g(.)$, the \emph{expected value of the random variable} $g(X)$ is written
\[
E_p \left(g(X)\right) = \sum_{x\in\mathcal{X}}g(x)p(x),
\]
or more simply as $E\left(g(X)\right)$ when the probability density function is understood from the context.

Let $X$ be a random variable which drawn according to probability density function $p(x)$. The \emph{entropy}\index{entropy!entropy} of $X$, $H(X)$ is defined as the expected value of the random variable $\log\frac{1}{p(x)}$, therefore, we have
\[
H(X) = -\sum_{x\in\mathcal{X}}p(x)\log p(x).
\]
The $\log$ is to the base 2 and entropy is expressed in bits. Furthermore, $0\log 0 =0$. Since $0\leq p(x)\leq 1$, we have $\log\frac{1}{p(x)}\geq 0$ which implies that $H(X)\geq 0$. Let us recall our coin toss example where the coin is not necessarily fair. That is denoting the event \emph{head} by \emph{H} and the event \emph{tail} by \emph{T}, let $P(H) = p$ and $P(T) = 1 - p$. Then the corresponding random variable $X$ is defined as $X(H) = 1$ and $X(T)=0$. That is we have

\begin{equation}
X = \left\{ \begin{array}{rcl}
1, & & \mathrm{Pr}\{X=1\}=p ;\\
0, &&  \mathrm{Pr}\{X=0\}=1 -  p.\nonumber
\end{array}\right.
\end{equation}

Then,
\[
H(X) = -p\log p - (1 - p)\log(1 - p).
\]
Note that the maximum of $H(X) $ is equal to $1$ which is attained when $p = \frac{1}{2}$. Thus, the entropy of a fair coin toss, i.e., $P(H) = P(T) = \frac{1}{2}$ is 1 bit. More generally for any random variable $X$,
\begin{equation}
H(X)\leq\log|\mathcal{X}|,\label{eq:maxent}
\end{equation}
with equality if and only if $X$ is uniformly distributed.

The \emph{joint entropy}\index{entropy!joint entropy} $H(X,Y)$ of a pair of discrete random variables $(X,Y)$ with a joint probability density function $p(x,y)$ is defined as
\[
H\left(X,Y\right) = -\sum_{x\in\mathcal{X}}\sum_{y\in\mathcal{Y}}p(x,y)\log p(x,y).
\]

Note that we can also express $H(X,Y)$ as
\[
H(X,Y) = -E \left(\log p(X,Y)\right).
\]
We can also define the conditional entropy of a random variable given another. The \emph{Conditional Entropy}\index{entropy!conditional entropy} $H(Y|X)$ is defined as
\begin{equation}\label{eq:entcond}
H(Y|X) = \sum_{x\in\mathcal{X}} p(x) H(Y|X=x).
\end{equation}

Now we can again get another description of the conditional entropy in terms of the conditional expectation of random variable as follows.
\begin{eqnarray}
H\left(Y|X\right)&=&\sum_{x\in\mathcal{X}}p(x)H\left(Y|X=x\right)\nonumber\\
&=&-\sum_{x\in\mathcal{X}}p(x)\sum_{y\in\mathcal{Y}}p(y|x)\log p(y|x)\nonumber\\
&=&-\sum_{x\in\mathcal{X}}\sum_{y\in\mathcal{Y}}p(x,y)\log p(y|x)\nonumber\\
&=&-E\log p(Y|X).\nonumber
\end{eqnarray}

The following theorem is proved by T. Cover and J. Thomas in \cite{Cov} pages 17 and 18.
\begin{theorem}
Let $X$, $Y$, and $Z$ be random variables with joint probability distribution $p(x,y,z)$. Then we have
\begin{eqnarray}
&&H\left(X,Y\right) = H\left(X\right) + H\left(Y|X\right),\nonumber\\
&&H\left(X,Y|Z\right) = H\left(X|Z\right) + H\left(Y|X,Z\right).\nonumber
\end{eqnarray}
 \qed
\end{theorem}

Furthermore, letting $f(.)$ be any function (see T. Cover and J. Thomas \cite{Cov} pages 34 and 35), we have
\[
0\leq H(X|Y)\leq H(X|f(Y))\leq H(X).
\]
\section{Relative Entropy and Mutual Information}
Let $X$ be a random variable and consider two different probability density functions $p(x)$ and $q(x)$ for $X$.  The \emph{relative entropy}\index{entropy!relative entropy|(} $D(p||q)$ is a measure of the distance between two distributions $p(x)$ and $q(x)$. The \emph{relative entropy}\index{entropy!relative entropy|)} or \emph{Kullback-Leibler distance}\index{distance!Kullback-Leibler distance} between two probability densities $p(x)$ and $q(x)$ is defined as
\begin{eqnarray}
D(p||q)=\sum_{x\in\mathcal{X}}p(x)\log \frac{p(x)}{q(x)},
\end{eqnarray}
We can see that $D(p||q) = E_p\log\frac{p(X)}{q(X)}$.

Now consider two random variables $X$ and $Y$ with a joint probability densities $p(x,y)$ and marginal densities $p(x)$ and $p(y)$. The \emph{mutual information}\index{mutual information!mutual information} $I(X;Y)$ is the relative entropy between the joint distribution and the product distribution $p(x)p(y)$. More precisely, we have

\begin{eqnarray}
I(X;Y)&&=\sum_{x\in\mathcal{X}}\sum_{y\in\mathcal{Y}}p(x,y)\log\frac{p(x,y)}{p(x)p(y)}\nonumber\\
&&=D(p(x,y)||p(x)p(y)).\nonumber
\end{eqnarray}

It is proved in T. Cover and J. Thomas \cite{Cov}, on pages 28 and 29, that we have
\begin{eqnarray}\label{eq:mut}
&&I(X;Y)=H(X) - H(X|Y),\\
&&I(X;Y)=H(Y) - H(Y|X),\nonumber\\
&&I(X;Y) =H(X) + H(Y) - H(X,Y),\nonumber\\
&&I(X;Y)=I(Y;X),\nonumber\\
&&I(X;X)=H(X).\nonumber
\end{eqnarray}

\section{Entropy and Counting}

In this section we consider the application of entropy method in counting problems. The following lemmas are two examples of using entropy methods in sovling well-known combinatorial probelms (see J. Radhakrishnan \cite{Rad}).

\begin{lemma}
\emph{(Shearer's Lemma)}. Suppose $n$ distinct points in $\mathbb{R}^3$ have $n_1$ distinct projections on the $XY$-plane, $n_2$ distinct projections on the $XZ$-plane and $n_3$ distinct projections on the $YZ$-plane. Then, $n^2\leq n_1n_2n_3$.

\proof

Let $P = \left(A,B,C\right)$ be one of the $n$ points picked at random with uniform distribution, and $P_1 = \left(A,B\right)$, $P_2 = \left(A,C\right)$, and $P_3=\left(B,C\right)$ are its three projections. Then we have

\begin{equation}
H\left(P\right) = H\left(A\right) + H\left(B|A\right) + H\left(C|A,B\right),\label{eq:n2}
\end{equation}

Furthermore,
\begin{eqnarray}
&&H\left(P_1\right) = H\left(A\right) + H\left(B|A\right),\nonumber\\
&&H\left(P_2\right) = H\left(A\right) + H\left(C|A\right),\nonumber\\
&&H\left(P_3\right) = H\left(B\right) + H\left(C|B\right).\nonumber
\end{eqnarray}
Adding both sides of these equations and considering \ref{eq:n2}, we have $2H\left(P\right)\leq H\left(P_1\right) + H\left(P_2\right) + H\left(P_3\right)$. Now, noting that $H\left(P\right) = \log n$, and $H\left(P_i\right)\leq \log n_i$, the lemma is proved.
\qed
\end{lemma}

As another application of the entropy method, we can give an upper bound on the number of the perfect matchings of a bipartite graph (see J. Radhakrishnan \cite{Rad}).
\begin{theorem}
\emph{(Br\'{e}gman's Theorem).}Let $G$ be a bipartite graph with parts $V_1$ and $V_2$ such that $|V_1| = |V_2| = n$. Let $d(v)$ denote the degree of a vertex $v$ in $G$. Then, the number of perfect matchings in $G$ is at most
\[
\prod_{v\in V_1}\left(d(v)!\right)^\frac{1}{d(v)}.
\]
\end{theorem}
\proof
Let $\mathcal X$ be the set of perfect matchings of $G$. Let $X$ be a random variable corresponding to the elements of $\mathcal X$ with uniform density. Then
\[
H(X) = \log |\mathcal X|.
\]
The following remark is useful in our discussion. Let $Y$ be any random variable with the set of possible values $\mathcal Y$. First note that the conditional entropy $H\left(Y|X\right)$ is obtained using (\ref{eq:entcond}). 
Let $\mathcal Y_x$ denote the set of possible values for the random variable $Y$ given $x\in\mathcal X$, that is
\[
\mathcal Y_x = \{y\in\mathcal Y: P(Y = y|X=x)>0\}.
\]
We partition the set $\mathcal X$ into sets $\mathcal X_1,\mathcal X_2,\cdots,\mathcal X_r$ such that for $i = 1,2,\cdots, r$ and all $x\in \mathcal X_i$, we have
\begin{equation}\label{eq:condY}
|\mathcal Y_x|=i.
\end{equation}
Letting $Y_x$ be a random variable taking its value on the set $\mathcal Y_x$ with uniform density, and noting equations (\ref{eq:maxent}) and (\ref{eq:condY}) for all $x\in\mathcal X_i$ we have
\begin{equation}\label{eq:cond1}
H\left(Y_x\right)=\log i.
\end{equation}
But note that
\begin{eqnarray}\label{eq:cond2}
H\left(Y|X\right) &=& E_X\left(H\left(Y_x\right)\right).
\end{eqnarray}
Then using (\ref{eq:cond1}) and (\ref{eq:cond2}), we get
\begin{eqnarray}\label{eq:cond3}
H\left(Y|X\right)\leq\sum_i^rP\left(X\in \mathcal X_i\right)\log i.
\end{eqnarray}
We define the random variable $X(v)$ for all $v\in V_1$ as
\[
X(v) := u~\text{such that}~u\in V_2~\text{and}~u~\text{is matched to}~v~\text{in}~X,~\forall v\in V_1.
\]
For a fixed ordering vertices $v_1,\cdots,v_n$ of $V_1$
\begin{eqnarray}
\log|\mathcal X| &=& H(X)\nonumber\\
                          &=& H\left(X(v_1)\right) + H\left(X(v_2)|X(v_1)\right) + \cdots + H\left(X(v_n)|X(v_1),\cdots,X(v_{n-1})\right)\label{eq:count1}
\end{eqnarray}
Now, pick a random permutation
\[
\tau:[n]\rightarrow V_1,
\]
and consider $X$ in the order determined by $\tau$. Then for every permutation $\tau$, we have
\[
H(X) = H\left(X(\tau(1))\right) + H\left(X(\tau(2))|X(\tau(1))\right)+\cdots+H\left(X(\tau(n))|X(\tau(1)),\cdots,X(\tau(n-1))\right).\label{eq:count2}
\]
By averaging over all $\tau$, we get
\[
H(X) = E_\tau\left(H\left(X(\tau(1))\right) + H\left(X(\tau(2))|X(\tau(1))\right)+\cdots+H\left(X(\tau(n))|X(\tau(1)),\cdots,X(\tau(n-1))\right)\right).\label{eq:count2}
\]
For a fixed $\tau$, fix $v\in V_1$ and let $k = \tau^{-1}(v)$. Then we let $\mathcal Y_{v,\tau}$ to be the set of vertices $u$ in $V_2$ which are adjacent to vertex $v\in V_1$ and 
\[
u\notin\{x\left(\tau(1)\right),x\left(\tau(2)\right),\cdots,x\left(\tau(k-1)\right)\}
\]
Letting $\mathcal N(v)$ be the set of neighbours of $v\in V_1$ in $V_2$, we have
\[
\mathcal Y_{v,\tau} = \mathcal N(v)\setminus\{x\left(\tau(1)\right),x\left(\tau(2)\right),\cdots,x\left(\tau(k-1)\right)\}.
\]
Letting $d(v)$ be the degree of vertex $v$ and $Y_{v,\tau} = |\mathcal Y_{v,\tau}|$ be a random variable taking its value in $\{1,\cdots,d(v)\}$, that is
\[
Y_{v,\tau} = j,~\text{for}~j\in\{1,\cdots,d(v)\}.
\]
Using (\ref{eq:cond2}) and noting that $P_{X(v),\tau}(Y_{v,\tau} = j) = \frac{1}{d(v)}$, we have
\begin{eqnarray}
H\left(X\right) &=&\sum_{v\in V_1}E_{\tau}\left(X(v)|X(\tau(1)),X(\tau(2)),\cdots,X(\tau(k-1))\right)\nonumber\\
&\leq&\sum_{v\in V_1}E_{\tau}\left(\sum_{j=1}^{d(v)}P_{X(v)}\left(Y_{v,\tau} = j\right).\log j\right)\nonumber\\
&=&\sum_{v\in V_1}\sum_{j=1}^{d(v)}E_{\tau}\left(P_{X(v)}\left(Y_{v,\tau} = j\right)\right).\log j\nonumber\\
&=&\sum_{v\in V_1}\sum_{j=1}^{d(v)}P_{X(v),\tau}\left(Y_{v,\tau} = j\right).\log j\nonumber\\
&=&\sum_{v\in V_1}\sum_{j=1}^{d(v)}\frac{1}{d(v)}\log j\nonumber\\
&=&\sum_{v\in V_1}\log\left(d(v)!\right)^{\frac{1}{d(v)}}.\nonumber 
\end{eqnarray} 
Then using (\ref{eq:count1}), we get
\[
|\mathcal X|\leq \left(d(v)!\right)^{\frac{1}{d(v)}}.
\]
\qed

%\newpage\null\thispagestyle{empty}\newpage
\chapter{Graph Entropy}

In this chapter, we introduce and study the entropy of a graph which was defined in \cite{JKor} by J. K\"{o}rner in 1973. We present several equivalent definitions of this parameter. However, we will focus mostly on the combinatorial definition which is going to be the main theme of this thesis.
\section{Entropy of a Convex Corner\index{entropy!entropy of a convex corner}}

A subset $\mathcal{A}$ of $\mathbb{R}_{+}^n$ is called a \emph{convex corner}\index{convex!convex corner|(} if it is compact, convex, has  non-empty interior, and for every $\mathbf a\in \mathcal{A}$, $\mathbf a^\prime\in \mathbb{R}_{+}^n$ with $\mathbf a^\prime\leq \mathbf a$, we have $\mathbf a^\prime \in \mathcal{A}$. For example, the \emph{vertex packing polytope}\index{polytope!vertex packing polytope} $VP(G)$ of a graph $G$, which is the convex hull of the characteristic vectors of its independent sets, is a convex corner.

Now, let $\mathcal{A}\subseteq \mathbb{R}_{+}^n$ be a convex corner, and $P\in \mathbb{R}_{+}^n$ a probability density, i.e., its coordinates add up to 1. The entropy of $P$ with respect to $\mathcal{A}$ is
\[
H_{\mathcal{A}}\left(P\right) = \min_{a\in \mathcal{A}} \sum_{i = 1}^np_i\log \frac{1}{a_i}.
\]
\begin{remark}\label{rem:Rem1}
Note that the function $-\sum_{i=1}^kp_i\log a_i$ in the definition of a convex corner\index{convex!convex corner|)} is a convex function and tends to infinity at the boundary of the non-negative orthant and tends monotonically to $-\infty$ along the rays from the origin.
\end{remark}
Consider the convex corner $\mathcal{S} := \{x\geq0 , \sum_i x_i \leq 1\}$, which is called a \emph{unit corner}\index{unit corner}. The following lemma relates the entropy of a random variable defined in the previous chapter to the entropy of the unit corner.
\begin{lemma}
The entropy $H_\mathcal{S}\left(P\right)$ of a probability density\index{probability!probability density} $P$ with respect to the unit corner $\mathcal{S}$ is just the regular (Shannon) entropy $H\left(P\right) = -\sum_{i}p_i\log p_i$.
\end{lemma}
\proof

From Remark \ref{rem:Rem1}, we have
\[
H_{\mathcal S}(\mathbf p) = \min_{\mathbf s\in\mathcal S} -\sum_i p_i\log s_i = \min_{\mathbf s\in\{x\geq0,~\sum_ix_i = 1\}}-\sum_i p_i\log s_i
\]
Thus the above minimum is attained by a probability density vector $\mathbf s$. More precisely, we have
\[
H_{\mathcal S}(\mathbf p) = D(\mathbf p||\mathbf s) + H(\mathbf p).\label{eq:unc}
\]
Noting that $ D(\mathbf p||\mathbf s) \geq 0$ and $ D(\mathbf p||\mathbf s) = 0$ if and only if $\mathbf s=\mathbf p$, we get
\[
H_{\mathcal S}(\mathbf p) = H(\mathbf p).
\]
\qed

There is another way to obtain the entropy of a convex corner\index{entropy!entropy of a convex corner}. Consider the mapping $\Lambda : \mathrm{int}~\mathbb{R}_{+}^n \rightarrow \mathbb{R}^n$ defined by

\[
\Lambda(x) := \left(-\log x_1,\cdots,-\log x_n\right).
\]

It is easy to see using the concavity of the log function that if $\mathcal{A}$ is a convex corner, then $\Lambda(\mathcal{A})$ is a closed, convex, full-dimensional set, which is \emph{up-monotone}, i.e., $a\in\Lambda(\mathcal{A})$ and $a^\prime\geq a$ imply $a^\prime\in\Lambda(\mathcal{A})$. Now, $H_\mathcal{A}(P)$ is the minimum of the linear objective function $\sum_{i}p_ix_i$ over $\Lambda(\mathcal{A})$. Now we have the following lemma (See \cite{Csis}).

\begin{lemma}\emph{(I. Csisz\'{a}r, J. K\"{o}rner, L. Lov\'{a}s , K. Marton, and G. Simonyi ).}\label{lem:corner1}
For two convex corners $\mathcal{A},\mathcal{C}\subseteq \mathbb{R}_{+}^k$, we have $H_\mathcal{A}(P)\geq H_\mathcal{C}(P)$ for all $P$ if and only if $\mathcal{A}\subseteq \mathcal{C}$.
\end{lemma}

\proof

The ``if" part is obvious. Assume that $H_\mathcal{C}(P)\leq H_\mathcal{A}(P)$ for all $P$. As remarked above, we have
\[
H_\mathcal{A}(P) = \min\{P^T\mathbf x :\mathbf x\in \Lambda(\mathcal{A})\},
\]
and hence it follows that we must have $\Lambda(\mathcal{A})\subseteq \Lambda(\mathcal{C})$. This clearly implies $\mathcal{A}\subseteq \mathcal{C}$.
\qed

Then we have the following corollary.

\begin{corollary}
We have $0\leq H_\mathcal{A}(P) \leq H(P)$ for every probability distribution $P$ if and only if $\mathcal{A}$ contains the unit corner and is contained in the unit cube.
\end{corollary}

\section{Entropy of a Graph\index{entropy!entropy of a graph}}

Let $G$ be a graph on vertex set
$V(G)=\{1,\cdots,n\}$, let $P=(p_1,\cdots,p_n)$ be a probability density on
$V(G)$, and let $VP(G)$ denote the vertex packing polytope of $G$. The \emph{entropy of $G$ with respect to $P$} is then defined as
\[
H_k(G,P) = \min_{\mathbf{a}\in VP(G)} \sum_{i=1}^n p_i\log (1/a_i).\label{eq:combent}
\]
Let $G = (V,E)$ be a graph with vertex set $V$ and edge set $E$. Let $V^n$ be the set of sequences of length $n$ from $V$. Then the graph $G^{(n)} = (V^{n}, E^{(n)})$ is the \emph{$n$-th conormal
power}. Two distinct vertices $x$ and $y$ of $G^{(n)}$ are adjacent in $G^{(n)}$ if there is some $i\in n$ such that $x_i$ and $y_i$ are adjacent in $G$, that is
\[
E^{(n)} = \{(x,y) \in V^n \times V^n :  \exists i :
(x_i,y_i)\in E\}.
\]
For a graph $F$ and $Z\subseteq V(F)$ we denote by $F[Z]$ the induced subgraph of $F$ on $Z$. The chromatic number of $F$ is denoted by $\chi(F)$.

Let
\[
T_\epsilon^{(n)} = \{U\subseteq V^n: P^n(U)\geq 1 - \epsilon\}.
\]
We define the functional $H(G,P)$ with respect to the probability distribution $P$ on the vertex set $V(G)$ as follows.
\begin{equation}\label{eq:ent1}
H(G,P) = \lim_{n\rightarrow\infty} \min_{U\in T_\epsilon^{(n)}}  \frac{1}{n} \log\chi(G^{(n)}[U]).
\end{equation}
Let $X$ and $Y$ be two discrete random variables taking their values on some (possibly different) finite sets and consider the random variable formed by the pair $(X,Y)$.

Now let $X$ denote a random variable taking its values on the vertex set of $G$ and $Y$ be a random variable
taking its values on the independent sets of $G$. Having a fixed distribution $P$ over the vertices, the set of feasible joint distributions $\mathcal Q$ consists of the joint distributions $Q$ of $X$ and $Y$ such that
\[
\sum_{y\in\mathcal Y}Q(X,Y = y) = P(X).
\]
As an example let the graph $G$ be a 5-cycle $C_5$ with the vertex set
\[
V(C_5) = \{x_1,x_2,x_3,x_4,x_5\},
\]
and let $\mathcal Y$ denote the set of independent sets of $G$. Let $P$ be the uniform distribution over the vertices of $G$, i.e.,
\[
P(X=x_i) = \frac{1}{5},~\forall i\in\{1,\cdots,5\},
\]
Noting that each vertex of $C_5$ lies in two maximal independent sets, we define the joint distribution $Q$ as
\begin{equation}
Q(X=x,Y=y) = \left\{ \begin{array}{rcl}
\frac{1}{10}, & &~\text{ $y$ maximal and $y\ni x$} ,\\
0, && ~\text{Otherwise}.
\end{array}\right.
\end{equation}
is a feasible joint distribution.

Now given a graph $G$, we define the functional $H^\prime(G,P)$ with respect to the probability distribution $P$ on the vertex set $V(G)$, as
\begin{equation}\label{eq:ent2}
H^\prime(G,P) = \min_{\mathcal{Q}} I(X;Y).
\end{equation}
%J. Korner in \cite{JKor}, I. Csiszar and et al. in \cite{Csis}, and G. Simonyi in \cite{Sim} proved the following lemma.
The following lemmas relate the functionals defined above.

\begin{lemma}\emph{(I. Csisz\'{a}r, et.~al.)}.\label{lem:equiv}
For every graph $G$ we have $H_k\left(G,P\right) = H^\prime\left(G,P\right)$.
\end{lemma}

\proof
%See \cite{JKor}, \cite{Sim}, and \cite{Csis}.
First, we show that $H_k\left(G,P\right)= H^\prime\left(G,P\right)$. Let $X$ be a random variable taking its values on the vertices of $G$ with probability density $P=\left(p_1,\cdots,p_n\right)$. Furthermore, let $Y$ be the random variable associated with the independent sets of $G$ and $\mathcal F(G)$ be the family of independent sets of $G$. Let $q$ be the conditional distribution of $Y$ which achieves the minimum in $(\ref{eq:ent2})$ and $r$ be the corresponding distribution of $Y$. Then we have
\[
H^\prime(G,P) = I\left(X;Y\right) = -\sum_ip_i\sum_{i\in F\in \mathcal F(G)}q\left(F|i\right)\log\frac{r(F)}{q(F|i)}.
\]
From the concavity of the $\log$ function we have
\[
\sum_{i\in F\in \mathcal F(G)}q\left(F|i\right)\log\frac{r(F)}{q(F|i)}\leq \log\sum_{i\in F\in \mathcal F(G)}r(F).
\]
Now we define the vector $\mathbf{a}$ by setting
\[
a_i = \sum_{i\in F\in \mathcal F(G)}r(F).
\]
Note that $\mathbf{a}\in VP(G)$. Hence,
\[
H^\prime\left(G,P\right)\geq -\sum_ip_i\log a_i.
\]
and consequently,
\[
H^\prime\left(G,P\right)\geq H_k\left(G,P\right).
\]
Now we prove the reverse inequality. Let $\mathbf{a}\in VP(G)$. Then letting $s$ be a probability density on $\mathcal F(G)$, we have
\[
a_i = \sum_{i\in F\in \mathcal F(G)}s(F).
\]
We define transition probabilities as
\begin{equation}
q(F|i) = \left\{ \begin{array}{rcl}
\frac{s(F)}{a_i} & & i\in F ,\\
0 && i\notin F.
\end{array}\right.
\end{equation}
Then, setting $r(F)=\sum_{i}p_iq(F|i)$, we get
\[
H^\prime(G,P)\leq \sum_{i,F}p_iq(F|i)\log\frac{q(F|i)}{r(F)}
\]
By the concavity of the $\log$ function, we get
\[
-\sum_{F}r(F)\log r(F)\leq -\sum_F r(F)\log s(F),
\]
Thus,
\[
-\sum_{i,F}p_iq(F|i)\log r(F)\leq -\sum_{i,F}p_iq(F|i)\log s(F).
\]
And therefore,
\[
H^\prime(G,P)\leq \sum_{i,F}p_iq(F|i)\log\frac{q(F|i)}{s(F)}=-\sum_ip_i\log a_i.
\]
\qed

\begin{lemma}\label{lem:equiv1}
\emph{(J. K\"{o}rner)}. For every graph $G$ we have $H^\prime\left(G,P\right) = H\left(G,P\right)$.
\end{lemma}
\proof

See  Appendix A.
\qed
\section{Graph Entropy and Information Theory}
A \emph{discrete memoryless} and \emph{stationary information source} $X$ is a sequence $\left\{X_i\right\}_{i=1}^\infty$ of independent, identically distributed discrete random variables with values in a finite set $\mathcal X$. Let $\mathcal X$ denote the set of the alphabet of a discrete memoryless and stationary information source with five elements. That is
\[
\mathcal X = \{x_1,x_2,x_3,x_4,x_5\}.
\]
We define a characteristic graph $G$ corresponding to $\mathcal X$ as follows. The vertex set of $G$ is
\[
V(G) = \mathcal X.
\]
Furthermore, two vertices of $G$ are adjacent if and only if the corresponding elements of $\mathcal X$ are distinguishable. As an example one can think of the 5-cycle of Figure \ref{fig:charcGraph1} as a characteristic graph of an information source $\mathcal X$.
\begin{figure}[!t]
    \begin {center}
        \begin{tikzpicture}
        [scale = 2]
            \draw (-.5,0) -- (-1,1);
            \draw (-1,1) -- (0,1.75);
            \draw (0,1.75) -- (1,1);
            \draw (1,1) -- (.5,0);
            \draw (.5,0) -- (-.5,0);
            \node[font=\small] at (0,2) {$x_1$};
            \node[font=\small] at (1.2,1.2) {$x_2$};
            \node[font=\small] at (.8,0) {$x_3$};
            \node[font=\small] at (-.8,0) {$x_4$};
            \node[font=\small] at (-1.3,1.3) {$x_5$};
            \filldraw [blue]
            (-0.5,0) circle (3pt)
            (-1,1) circle (3pt)
            (0,1.75) circle (3pt)
            (1,1) circle (3pt)
            (.5,0) circle (3pt);
        \end{tikzpicture}
    \end{center}
\caption{A characteristic graph of an information source with 5 alphabets}.\label{fig:charcGraph1}
  \end{figure}
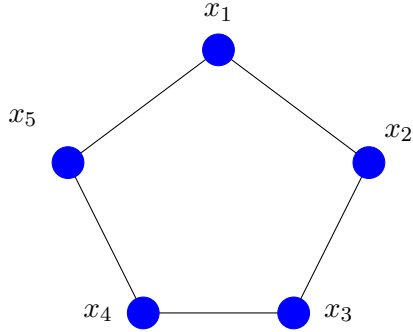
In the source coding problem, our goal is to label the vertices of the characteristic graph with minimum number of labels so that we can recover the elemnets of a given alphabet in a unique way. This means that we should colour the vertices of the graph properly with minimum number of colours. More precisely, one way of encoding the elements of the source alphabet $\mathcal X$ in Figure \ref{fig:charcGraph1} is
\begin{eqnarray}
\{x_1,x_3\}&\rightarrow& \text{red,}\nonumber\\
\{x_2,x_4\}&\rightarrow& \text{blue,}\nonumber\\
\{x_5\}&\rightarrow& \text{green.}\nonumber\\
\end{eqnarray}
Now, let $X$ be a random variable takes its values from $\mathcal X$ with the following probability density
\[
P(X=x_i)=p_i,~\forall i\in\{1,\cdots,5\}.
\]
Now consider the graph $G^{(n)}$, and let $\epsilon>0$. Then neglecting vertices of $G^{(n)}$ having a total probability less than $\epsilon$, the encoding of vertices of $G^{(n)}$ essentially becomes the colouring of a sufficiently large subgraph of $G^{(n)}$. And therefore, the minimum number of codewords is
\[
\min_{U\in T_\epsilon^{(n)}}  \chi(G^{(n)}(U)).
\]
Taking logarithm of the above quantity, normalizing it by $n$, and making $n$ very large, we get the minimum number of required information bits which is the same as the graph entropy of $G$. The characteristic graph of a regular source where distinct elements of the source alphabet are distinguishable is a complete graph. We will see in section 3.5 that the entropy of a complete graph is the same as the entropy of a random variable.

\section{Basic Properties of Graph Entropy}
The main properties of graph entropy are \emph{monotonicity}, \emph{sub-additivity}, and \emph{additivity under vertex substitution}. Monotonicity is formulated in the following lemma.

\begin{lemma}\emph{(J. K\"{o}rner).}\label{lem:mono}
Let $F$ be a spanning subgraph of a graph $G$. Then for any probability density $P$ we have $H_k(F,P)\leq H_k(G,P)$.
\end{lemma}

\proof

For graphs $F$ and $G$ mentioned above, we have $VP(G)\subseteq VP(F)$. This immediately implies the statement by the definition of graph entropy.
\qed

The sub-additivity was first recognized by K\"{o}rner in \cite{JKor1} and he proved the following lemma.
\begin{lemma}\label{lem:subadd}\emph{(J. K\"{o}rner).}
Let $F$ and $G$ be two graphs on the same vertex set $V$ and $F\cup G$ denote the graph on $V$ with edge set $E(F)\cup E(G)$. For any fixed probability density $P$ we have
\[
H_k\left(F\cup G,P\right) \leq H_k\left(F,P\right) + H_k\left(G,P\right).
\]
\end{lemma}

\proof

Let $\mathbf{a}\in VP(F)$ and $\mathbf{b}\in VP(G)$ be the vectors achieving the minima in the definition of graph entropy for $H_k\left(F,P\right)$ and $H_k\left(G,P\right)$, respectively. Notice the vector $\mathbf{a} \circ \mathbf{b} = (a_1b_1, a_2b_2, \cdots, a_nb_n)$ is in $VP(F\cup G)$, simply because the intersection of a stable set of $F$ with a stable set of $G$ is always a stable set in $F\cup G$. Hence, we have

\begin{align*}
H_k\left(F,P\right) + H_k\left(G,P\right) &= \sum_{i = 1}^np_i\log\frac{1}{a_i} + \sum_{i = 1}^np_i\log\frac{1}{b_i} \\
&= \sum_{i = 1}^np_i\log\frac{1}{a_ib_i}\\
&\geq  H_k\left(F\cup G, P\right).
\end{align*}
\qed

\begin{figure}[!t]%
\centering
\subfloat[A 5-cycle $G$.]{
\begin{tikzpicture}
[scale = 2]
            \draw (-.5,0) -- (-1,1);
            \draw (-1,1) -- (0,1.75);
            \draw (0,1.75) -- (1,1);
            \draw (1,1) -- (.5,0);
            \draw (.5,0) -- (-.5,0);
            \node[font=\small] at (0,2) {$u_1$};
            \node[font=\small] at (1.2,1.2) {$u_2$};
            \node[font=\small] at (.8,0) {$u_3$};
            \node[font=\small] at (-.8,0) {$u_4$};
            \node[font=\small] at (-1.3,1.3) {$u_5$};
            \filldraw [blue]
            (-0.5,0) circle (3pt)
            (-1,1) circle (3pt)
            (0,1.75) circle (3pt)
            (1,1) circle (3pt)
            (.5,0) circle (3pt);
\end{tikzpicture}
}
\qquad\qquad\qquad
\subfloat[A triangle $F$.]{
\begin{tikzpicture}
[scale = 2]
            \draw (-.5,1.5) -- (.5,1.5);
            \draw (.5,1.5) -- (0,2);
            \draw (0,2) -- (-.5,1.5);
            \node[font=\small] at (0,2.2) {$v_1$};
            \node[font=\small] at (.7,1.7) {$v_2$};
            \node[font=\small] at (-.7,1.7) {$v_3$};
            \filldraw [blue]
            (-0.5,1.5) circle (3pt)
            (.5,1.5) circle (3pt)
            (0,2) circle (3pt);
\end{tikzpicture}
}
\qquad\qquad\qquad
\subfloat[The graph $G_{u_1\longleftarrow F}$]{
\begin{tikzpicture}
[scale = 2]
            \draw (-.5,0) -- (-1,1);
            \draw (-1,1) to [out=120,in=150] (0,2);
            \draw (-1,1) -- (-.5,1.5);
            \draw (-1,1) -- (.5,1.5);
            \draw (0,2) to [out=30,in=60]  (1,1);
            \draw (-.5,1.5) -- (1,1);
            \draw (.5,1.5) -- (1,1);
            \draw (1,1) -- (.5,0);
            \draw (.5,0) -- (-.5,0);
            \draw (-.5,1.5) -- (.5,1.5);
            \draw (-.5,1.5) -- (0,2);
            \draw (.5,1.5) -- (0,2);
            \node[font=\small] at (0,2.2) {$v_1$};
            \node[font=\small] at (-.7,1.7) {$v_2$};
            \node[font=\small] at (.7,1.7) {$v_3$};
            \node[font=\small] at (1.2,1.2) {$u_2$};
            \node[font=\small] at (.8,0) {$u_3$};
            \node[font=\small] at (-.8,0) {$u_4$};
            \node[font=\small] at (-1.3,1.3) {$u_5$};
            \filldraw [blue]
            (-0.5,0) circle (3pt)
            (-1,1) circle (3pt)
            (0,2) circle (3pt)
            (-.5,1.5) circle (3pt)
            (.5,1.5) circle (3pt)
            (1,1) circle (3pt)
            (.5,0) circle (3pt);
\end{tikzpicture}
}
\caption{}%
\label{fig:Subs}%
\end{figure}

The notion of substitution is defined as follows. Let $F$ and $G$ be two vertex disjoint graphs and $v$ be a vertex of $G$. By substituting $F$ for $v$ we mean deleting $v$ and joining every vertex of $F$ to those vertices of $G$ which have been adjacent with $v$. We will denote the resulting graph $G_{v\leftarrow F}$. We extend this concept also to distributions. If we are given a probability distribution $P$ on $V(G)$ and a probability distribution $Q$ on $V(F)$ then by $P_{v\leftarrow Q}$ we denote the distribution on $V\left(G_{v\leftarrow F}\right)$ given by $P_{v\leftarrow Q}(x) = P(x)$ if $x \in V(G) \setminus {v}$ and $P_{v\leftarrow Q}(x) = P(x) Q(x)$ if $x \in V(F)$. This operation is illustrated in Figure \ref{fig:Subs}.

Now we state the following lemma whose proof can be found in J. K\"{o}rner, et. al. \cite{JKor2}.

\begin{lemma}\emph{(J. K\"{o}rner, G. Simonyi, and Zs. Tuza).}\label{lem:Subs}
Let $F$ and $G$ be two vertex disjoint graphs, $v$ a vertex of $G$, while $P$ and $Q$ are probability distributions on $V(G) $ and $V(F)$, respectively. Then we have
\[
H_k\left(G_{v\leftarrow F}, P_{v\leftarrow Q}\right) = H_k\left(G,P\right) + P(v)H_k\left(F,Q\right).
\]
\qed
\end{lemma}

Notice that the entropy of an empty graph (a graph with no edges) is always zero (regardless of the distribution on its vertices). Noting this fact, we have the following corollary as a consequence of Lemma \ref{lem:Subs}.

\begin{corollary}\label{cor:EntrDisc}
Let the connected components of the graph $G$ be the subgraphs $G_i$ and $P$ be a probability distribution on $V(G)$. Set
\[
P_i(x) = P(x)\left(P(V(G_i))\right)^{-1}, x\in V(G_i).
\]
Then
\[
H_k\left(G,P\right) = \sum_i P\left(V(G_i)\right) H_k\left(G_i,P_i\right).
\]
\end{corollary}

\proof

Consider the empty graph on as many vertices as the number of connected components of $G$. Let a distribution be given on its vertices by $P\left(V(G_i)\right)$ being the probability of the vertex corresponding to the $i$th component of $G$. Now substituting each vertex by the component it belongs to and applying Lemma \ref{lem:Subs} the statement follows.
\qed

\section{Entropy of Some Special Graphs}
Now we look at entropy of some graphs which are also mentioned in G. Simonyi \cite{Sim} and \cite{Simu} . The first one is the complete graph.

\begin{lemma}
For $K_n$, the complete graph on $n$ vertices, one has
\[
H_k\left(K_n,P\right) = H(P).
\]
\end{lemma}

\proof
By definition of entropy of a graph, $H_k\left(K_n,P\right)$ has the form $\sum_{i=1}^np_i\log\frac{1}{q_i}$ where $q_i\geq 0$ for all $i$ and $\sum_{i=1}^n q_i =1$. This expression is well known to take its minimum at $q_i = p_i$. Indeed, by the concavity of the log function $\sum_{i=1}^n p_i\log\frac{p_i}{q_i}\leq \log\sum_{i=1}^nq_i = 0$.
\qed

And the next one is the complete multipartite graph.

\begin{lemma}\label{lem:components}
Let $G = K_{m_1,m_2,\cdots,m_k}$, i.e., a complete $k$-partite graph with maximal stable sets of size $m_1,m_2,\cdots,m_k$. Given a distribution $P$ on $V(G)$ let $Q$ be the distribution on $S(G)$, the set of maximal stable sets of $G$, given by $Q(J) = \sum_{x\in J}P(x)$ for each $J\in S(G)$. Then $H_k(G,P) = H_k\left(K_k,Q\right)$.
\end{lemma}

\proof
The statement follows from Lemma \ref{lem:Subs} and substituting stable sets of size \\$m_1,m_2,\cdots,m_k$ for the vertices of $K_k$.
\qed

A special case of the above Lemma is the entropy of a complete bipartite graph with equal probability measure on its stable sets equal to 1.
Now, let $G$ be a bipartite graph with color classes $A$ and $B$. For a set $D\subseteq A$, let $\mathcal{N}(D)$ denotes the the set of neighbours of $D$ in $B$, that is a subtes of the vertices in $B$ which are adjacent to a vertex in $A$.

Given a distribution $P$ on $V(G)$ we have
\[
P(D) = \sum_{i\in D}p_i~~\forall D\subseteq V(G),
\]
Furthermore, defining the binary entropy as
\[
h(x) := -x\log x - (1 - x)\log (1 - x),~~0\leq x\leq 1,
\]
J. K\"{o}rner and K. Marton proved the following theorem in \cite{JKor3}.

\begin{theorem}\emph{(J. K\"{o}rner and K. Marton).}\label{thm:bipentropy}
Let $G$ be a bipartite graph with no isolated vertices and $P$ be a probability distribution on its vertex set. If
\[
\frac{P(D)}{P(A)} \leq \frac{P(\mathcal{N}(D))}{P(B)},
\]
for all subsets $D$ of $A$, then
\[
H_k\left(G,P\right) = h\left(P(A)\right).
\]
And if
\[
\frac{P(D)}{P(A)} > \frac{P(\mathcal{N}(D))}{P(B)},
\]
then there exists a partition of $A = D_1\cup\cdots\cup D_k$ and a partition of $B = U_1\cup\cdots\cup U_k$ such that
\[
H_k\left(G,P\right) = \sum_{i=1}^k P\left(D_i \cup U_i\right) h\left(\frac{P(D_i)}{P(D_i\cup U_i)}\right).
\]
\end{theorem}
\proof

Let us assume the condition in the theorem statement holds. Then, using max-flow min-cut theorem (see A. Schrijver \cite{Schriv1} page 150), we show that there exists a probability density $Q$ on the edges of $G$ such that for all vertices $v\in A$, we have
\begin{equation}\label{eq:edgedensity}
\sum_{v\in e\in E(G)}Q(e) = \frac{p(v)}{P(A)},
\end{equation}
We define a digraph $D^\prime $ by
\[
V(D^\prime) = V(G)\cup\{s,t\},
\]
and joining vertices $s$ and $t$ to all vertices in parts $A$ and $B$, respectively. The edges between $A$ and $B$ are the exactly the same edges in $G$. Furthermore, we orient edges from $s$ toward $A$ and from $A$ toward $B$ and from $B$ to $t$. We define a capacity function $c:E(D^\prime)\rightarrow\mathbb R_+$ as
\begin{equation}
c(e) = \left\{ \begin{array}{rcl}
\frac{p(v)}{P(A)}, & & e = (s,v),~v\in A ,\\
1, &&  e=(v,u),~v\in A~\text{and}~u\in B,\\
\frac{p(u)}{P(B)},&& e=(u,t),~u\in B.
\end{array}\right.
\end{equation}
By the definition of $c$, we note that the maximum $st$-flow is at most $1$. Now, by showing that the minimum capacity of an $st$-cut is at least $1$, we are done.

Let $\delta(U)$ be a $st$-cut for some subset $U = \{s\}\cup A^\prime \cup B^\prime$ of $V(D^\prime)$ with $A^\prime\subseteq A$ and $B^\prime\subseteq B$.  If
\[
\mathcal N\left(A^\prime\right)\nsubseteq B^\prime,
\]
then
\[
c\left(\delta(U)\right)\geq 1.
\]
So suppose that
\[
\mathcal N\left(A^\prime\right)\subseteq B^\prime.
\]
Then using the assumption
\[
\frac{P(A^\prime)}{P(A)}\leq \frac{P\mathcal N(A^\prime)}{P(A)},
\]
we get
\begin{eqnarray}\label{eq:mincut}
c\left(\delta(U)\right)&\geq&\frac{P(B^\prime)}{P(B)}+\frac{P(A\setminus A^\prime)}{P(A)}\nonumber\\
&\geq&\frac{P(A^\prime)}{P(A)}+\frac{P(A\setminus A^\prime)}{P(A)} = 1.
\end{eqnarray}
Now, we define the vector $\mathbf b \in\mathbb R_+^{|V(G)|}$, as follows,
\[
\left(\mathbf b\right)_v := \frac{p(v)}{P(A)}.
\]
Then using (\ref{eq:edgedensity}), we have
\[
\mathbf b\in VP\left(\overline{G}\right),
\]
Thus,
\[
H_k(\overline{G},P)\leq\sum_{v\in V(G)}p(v)\log\frac{1}{b_v} = H(P) - h\left(P(A)\right).
\]
Then, using Lemma \ref{lem:mono} and Lemma \ref{lem:components}, we have
\[
H_k(G,P)\leq h(P(A)),
\]
Now, adding the last two inequalities we get
\begin{equation}\label{eq:inequality1}
H_k\left(G,P\right) +H_k\left(\overline{G},P\right) \leq H(P).
\end{equation}
On the other hand, by Lemma \ref{lem:subadd}, we also have
\begin{equation}\label{eq:inequality2}
H(P)\leq H_k\left(G,P\right) + H_k\left(\overline{G},P\right),
\end{equation}
Comparing (\ref{eq:inequality1}) and (\ref{eq:inequality2}), we get
\[
H(P)= H_k\left(G,P\right) + H_k\left(\overline{G},P\right),
\]
which implies that
\[
H_k(G,P) = h(P(A)).
\]
This proves the first part of the theorem.

Now, suppose that the condition does not hold. Let $D_1$ be a subset of $A$ such that
\[
\frac{P(D_1)}{P(A)}.\frac{P(B)}{P(\mathcal N(D_1))}
\]
is maximal. Now consider the subgraph $(A \setminus D_1)\cup(B\setminus\mathcal N(D_1))$ and for $i=2,\cdots,k$ let
\[
D_i\subseteq A\setminus \bigcup_{j=1}^{i-1}D_j,
\]
such that
\[
\frac{P(D_i)}{P(A \setminus \bigcup_{j=1}^{i-1}D_j)}.\frac{P(B\setminus\bigcup_{j=1}^{i-1}\mathcal N(D_j))}{P(\mathcal N(D_i))},
\]
is maximal.
Let us
\[
U_i = \mathcal N(D_i) \setminus \mathcal N(D_i\cup\cdots\cup D_{i-1}),\quad\text{for}~i=1,\cdots,k.
\]
Consider the independent sets $J_0,\cdots,J_k$ of the following form
\[
J_0 = B,~J_1 = D_1\cup B \setminus U_1,\cdots,~J_i =D_1\cup\cdots\cup D_i\cup B \setminus U_1\setminus\cdots\setminus U_i,\cdots,~J_k = A.
\]
Set
\begin{eqnarray}
&&\alpha(J_0) = \frac{P(U_1)}{P(U_1\cup D_1)},\nonumber\\
&&\alpha(J_i) = \frac{P(U_{i+1})}{P(U_{i+1}\cup D_{i+1})} - \frac{P(U_i)}{P(U_i\cup D_i)},\quad\text{for}~i=1,\cdots,k-1,\nonumber\\
&&\alpha(J_k) = 1 - \frac{P(U_k)}{P(U_k\cup D_k)}.\nonumber
\end{eqnarray}
Note that by the choice of $D_i$'s, all $\alpha(J_i)$'s are non-negative and add up to one. This implies that the vector $\mathbf a\in \mathbb{R}_+^{|V(G)|}$ defined as
\[
a_j = \sum_{j\in J_r}\alpha(J_r),\quad\forall j\in V(G),
\]
is in $VP(G)$.
Furthermore,
\begin{equation}
a_j = \left\{ \begin{array}{rcl}
\frac{P(D_i)}{P(D_i\cup U_i)}, & & j\in D_i ,\\
\frac{P(U_i)}{P(D_i\cup U_i)}, && j\in U_i.\nonumber
\end{array}\right.
\end{equation}
By the choice of the $D_j$'s and using the same max-flow min-cut argument we had, there exists a probability density $Q_i$ on edges of $G[D_i\cup U_i]$ such that
\begin{eqnarray}
&&b_j^\prime = \sum_{j\in e\in E\left(G[D_i\cup U_i]\right)} Q_i(e) = \frac{p_j}{P(D_i)},\quad\forall j\in D_i,\nonumber\\
&&b_j^\prime = \sum_{j\in e\in E\left(G[D_i\cup U_i]\right)} Q_i(e) = \frac{p_j}{P(U_i)},\quad\forall j\in U_i.\nonumber
\end{eqnarray}
Now we define the probability density $Q$ on the edges of $G$ as follows
\begin{equation}
Q(e) = \left\{ \begin{array}{rcl}
&P(D_i\cup U_i)Q_i(e),  & e\in E\left(G[D_i\cup U_i]\right) ,\\
&0,  & e\notin E\left(G[D_i\cup U_i]\right).\nonumber
\end{array}\right.
\end{equation}
The corresponding vector $\mathbf b\in VP\left(\overline G\right)$ is given by
\[
b_j = P\left(D_i\cup U_i\right)b_j^\prime,\quad\text{for}~j\in D_i\cup U_i.
\]
The vectors $\mathbf a\in VP\left(G\right)$ and $\mathbf b\in VP\left(\overline G\right)$ are the minimizer vectors in the definition of $H_k\left(G,P\right)$ and $H_k\left(\overline{G},P\right)$, respectively. Suppose that is not true. Then noting that the fact that by the definition of $\mathbf a$ and $\mathbf b$, we have
\[
\sum_{j\in V(G)}p_j\log\frac{1}{a_j} +\sum_{j\in V(G)}p_j\log\frac{1}{b_j} = \sum_{j\in V(G)}p_j\log\frac{1}{p_j} = H(P).
\]
the sub-additivity of graph entropy is violated. Now, it can be verified that $H_k\left(G,P\right)$ is equal to what stated in the theorem statement.
\qed
\section{Graph Entropy and Fractional Chromatic Number}

In this section we investigate the relation between the entropy of a graph and its fractional chromatic number\index{fractional chromatic number} which was already established by G. Simonyi \cite{Simu}. First we recall that the \emph{fractional chromatic number} of a graph $G$ is denoted by $\chi_f\left(G\right)$ is the minimum sum of nonnegative weights on the stable sets of $G$ such that for any vertex the sum of the weights on the stable sets of $G$ containing that vertex is at least one (see C. Godsil and G. Royle \cite{CGodsil}). I.Csisz\'{a}r and et.~al. \cite{Csis} showed that for every probability density $P$, the entropy of a graph $G$ is attained by a point $\mathbf a\in VP(G)$ such that there is not any other point $\mathbf a^\prime\in VP(G)$ majorizing the point $\mathbf a$ coordinate-wise. Furthermore, for any such point $\mathbf a\in VP(G)$ there is some probability density $P$ on $VP(G)$ such that the value of $H_k\left(G,P\right)$ is attained by $\mathbf a$.
Using this fact G. Simonyi \cite{Simu} proved the following lemma.
\begin{lemma}\emph{(G. Simonyi).}\label{lem:keylemma}
For a graph $G$ and probability density $P$ on its vertices with fractional chromatic number $\chi_f(G)$, we have
\[
\max_{P} H_k(G,P) = \log\chi_f(G).
\]
\end{lemma}
\proof
Note that for every graph $G$ we have $\left(\frac{1}{\chi_f(G)},\cdots,\frac{1}{\chi_f(G)}\right)\in VP(G)$. Thus for every probability density $P$, we have
\[
H_k\left(G,P\right)\leq \log \chi_f(G).
\]
Now, from the definition of the fractional chromatic number we deduce that graph $G$ has an induced subgraph $G^\prime$ with $\chi_f\left(G^\prime\right)=\chi_f\left(G\right)=\chi_f$ such that
\[
\forall \mathbf y\in VP\left(G^\prime\right),~\mathbf y\geq\frac{\mathbf 1}{\chi_f}~\text{implies}~\mathbf y=\frac{\mathbf 1}{\chi_f}.
\]
Now, by the above remark from I.Csisz\'{a}r and et.~al. \cite{Csis}, there exists a probability density $P^\prime$ on $VP(G^\prime)$ such that $H_k\left(G^\prime,P^\prime\right) = \log\chi_f$. Extending $P^\prime$ to a probability distribution $P$ as
\begin{equation}
p_i = \left\{ \begin{array}{rcl}
 p_i^\prime,& & i\in V(G) ,\\
0, && i\in V(G)-V(G^\prime).
\end{array}\right.
\end{equation}
the lemma is proved.
\qed

Now there is a natural question of uniqueness of the probability density which is a maximizer in the above lemma.
Using the above lemma we compute the fractional chromatic number of a vertex transitive graph in the following corollary.
\begin{corollary}
Let $G$ be a vertex transitive graph with $|V(G)|=n$, and let $\alpha(G)$ denote the size of a coclique of $G$ with maximum size. Then
\[
\chi_f(G) = \frac{n}{\alpha(G)}.
\]
\end{corollary}
\proof

First note that since $G$ is a vertex transitive graph, there exists a family of cocliques $S_1,\cdots,S_b$ of size $\alpha(G)$ that cover the vertex set of $G$, i.e., $V(G)$ uniformly. That is each vertex of $G$ lies in exactly $r$ of these cocliques, for some constant $r$. Thus we have
\begin{equation}
b\alpha(G) = nr,\label{eq:vxtrans}
\end{equation}
Now, we define a fractional coloring $\mathbf f$ as follows
\begin{equation}\label{eq:fraccol}
f_i = \left\{ \begin{array}{rcl}
\frac{1}{r}, & & i\in \{1,\cdots,b\} ,\\
0, &&~\text{Otherwise}.
\end{array}\right.
\end{equation}
Thus, from the definition of the fractional chromatic number of a graph, (\ref{eq:vxtrans}), and (\ref{eq:fraccol}), we have
\begin{equation}
\log \chi_f(G)\leq\log\sum_i f_i = \log\frac{b}{r} = \log\frac{n}{\alpha(G)}.\label{eq:ineq1}
\end{equation}
Now suppose that the probability density $\mathbf u$ of the vertex set $V(G)$ is uniform and let $\mathbf B$ be the $01$-matrix whose columns are the characteristic vectors of the independent sets in $G$. Then
\[
VP\left(G\right) = \{\mathbf x\in\mathbb R_+^n~:~\mathbf B\mathbf{\lambda} = \mathbf x,~\sum_i\lambda_i =1, ~\lambda_i\geq0,\forall i\}
\]
Consider the function
\[
g(\mathbf x) = -\frac{1}{n}\sum_{i=1}^n\log x_i.
\]
We want to minimize $g(\mathbf x)$ over $VP(G)$. So we use the vector $\mathbf{\lambda}$ in the definition of $VP(G)$ above. Furthermore, from our discussion above, note that each vertex of a vertex transitive graph lies in a certain number of independent sets $m$. Thus, we rewrite the function $g(.)$ in terms of $\mathbf{\lambda}$ as
\[
g(\lambda) = -\frac{1}{n}\log(\lambda_{i_1}+\cdots+\lambda_{i_m})-\cdots-\frac{1}{n}\log(\lambda_{j_1}+\cdots+\lambda_{j_m}).
\]
Now let $\mathcal S$ be the set of independent sets of $G$, and $\nu,\gamma_i\geq 0$
for all $i\in\{1,\cdots,|\mathcal S|\}$ be the Lagrange multipliers. Then the Lagrangian function $L_g(\nu,\gamma_1,\cdots,\gamma_{|\mathcal S|})$ is
\[
L_g(\nu,\gamma_1,\cdots,\gamma_{|\mathcal S|}) = g(\lambda) + \nu\left(\sum_{i=1}^{|\mathcal S|} \lambda_i-1\right) - \sum_{i}^{|\mathcal S|}\gamma_i\lambda_i,
\]
Now using Karush-Kuhn-Tucker conditions for our convex optimization problem (see S. Boyd and L. Vanderberghe\cite{Boyd}) we get
\begin{eqnarray}
&&\nabla L_g (\nu,\gamma_1,\cdots,\gamma_{|\mathcal S|}) = 0,\nonumber\\
&& \gamma_i\geq 0,~i\in\{1,\cdots,|\mathcal S|\},\nonumber\\
&& \gamma_i\lambda_i = 0,~i\in\{1,\cdots,|\mathcal S|\}.\label{eq:KKT}
\end{eqnarray}
Then considering the co-clique cover $\{S_1,\cdots,S_b\}$ above with $|S_i| = \alpha(G)$ for all $i$, one can verify that $\mathbf{\lambda}^*$ defined as
\begin{equation}
\lambda_i^* = \left\{ \begin{array}{rcl}
\frac{\alpha(G)}{nr}, & & i\in \{1,\cdots,b\} ,\\
0, &&~\text{Otherwise}.
\end{array}\right.
\end{equation}
is an optimum solution to our minimization problem. Since setting $\gamma_i =0$ for $i\in\mathcal S\setminus \{1,\cdots,b\}$ along with $\mathbf{\lambda^*}$ gives a solution to (\ref{eq:KKT}). Substituting $\mathbf{\lambda^*}$ into $g(\mathbf{\lambda})$
\[
H_k\left(G,U\right) = \log\frac{n}{\alpha(G)}.
\]
Using (\ref{eq:ineq1}) and Lemma \ref{lem:keylemma}, the corollary is proved.
%Let $\mathbf x^*$ minimizes the function $g(\mathbf x) = -\frac{1}{n}\sum_i\log x_i$ over $VP(G)$. Thus, the convex sets $VP(G)$ and $\{\mathbf x\in\mathbb R_+^{|V|}~:~g(\mathbf x)<g(\mathbf x^*)\}$ are disjoint, and therefore can be separated by a hyperplane. Furthermore, note that the two sets touch at $\mathbf x^*$ and the second one is smooth there. Thus the separating hyperplane is its tangent there. Let us define the vector $\mathbf h$ as
%\[
%h_i = \frac{1}{|V|x_i^*}~\forall i\in \{1,\cdots,n\}.
%\]
%Then the gradient of $-g(\mathbf x)$ at $\mathbf x^*$ is
%\[
%\mathbf h = \left(\frac{1}{|V|x_1^*},\cdots,\frac{1}{|V|x_n^*}\right),
%\]
%Noting that for every $\{\mathbf x\in\mathbb R_+^{|V|}~:~g(\mathbf x)<g(\mathbf x^*)\}$ and every $\mathbf x\in VP(G)$ we have $\mathbf h^T\mathbf x >1$ and $\mathbf h^T\mathbf x <= 1$, respectively. And for $\mathbf x^*$ we have
%\[
%\mathbf h^T\mathbf x^* = 1.
%\]
%the separating hyperplane is
%\[
%\mathbf h^T\mathbf x = 1.
%\]
%
%Now noting that $g(\mathbf x^*) = \log\frac{|V|}{\alpha(G)}$, using (\ref{eq:ineq1}) and Lemma \ref{lem:keylemma}, the corollary is proved.
\qed

The above corollary implies that the uniform probability density is a maximizer for $H_k\left(G,P\right)$ for a vertex transitive graph. We will give another proof of this fact at the end of the next chapter using \emph{chromatic entropy}\index{chromatic entropy}.

We have also the following corollary.
\begin{corollary}
For any graph $G$ and probability density $P$, we have
\[
H_k\left(G,P\right) \leq \log\chi(G).
\]
Equality holds if $\chi(G) = \chi_f(G)$ and $P$ maximizes the left hand side above.
\end{corollary}
Note that (\ref{eq:ent1}), Lemma \ref{lem:equiv}, Lemma \ref{lem:equiv1}, and the sub-multiplicative\index{sub-multiplicative} nature of the chromatic number, also results in the above corollary.

\section{Probability Density Generators}
For a pair of vectors $\mathbf a, \mathbf b\in \mathbb R_+^k$, $\mathbf a \circ \mathbf b$ denotes the
\emph{Schur product}\index{product!Schur product} of $\mathbf a$ and $\mathbf b$, i.e.,
\[
\left(\mathbf a\circ \mathbf b\right)_i = a_i.b_i,~i=1,\cdots,k.
\]
Then for two sets $\mathcal A$ and $\mathcal B$, we have
\[
\mathcal A\circ\mathcal B = \{\mathbf a\circ \mathbf b: \mathbf a\in\mathcal A,~\mathbf b\in\mathcal B\}.
\]
We say a pair of sets $\mathcal A,~\mathcal B\in \mathbb R_+^k$ is a \emph{generating pair}, if every probability density vector $\mathbf p\in \mathbb R_+^k$ can be represented as the schur product of the elements of $\mathcal A$ and $\mathcal B$, i.e.,
\[
\mathbf p=\mathbf a\circ \mathbf b,~\mathbf a\in \mathcal A,~\mathbf b\in \mathcal B.
\]
In this section we characterize a \emph{pair of generating convex corners}. First, we recall the definition of the \emph{antiblocker}\index{antiblocker} of a convex corner (see D. R. Fulkerson \cite{Flker}). The \emph{antiblocker} of a convex corner $\mathcal{A}$ is defined as
\[
\mathcal{A}^* :=\left\{\mathbf{b}\in \mathbb{R}_+^n : \mathbf{b}^T\mathbf{a}\leq 1,~~\forall \mathbf a\in \mathcal{A}\right\},
\]
which is itself a convex corner.

The following lemma relates entropy to antiblocking pairs (see I. Csisz\'{a}r and et.~al. \cite{Csis}).
\begin{lemma}\emph{(I. Csisz\'{a}r and et.~al.).}
Let $\mathcal A, \mathcal B \subseteq \mathbb{R}_+^n$ be convex corners and $\mathbf p\in \mathbb{R}_+^n$ a probability density. Then
\begin{description}
\item[(i)] If $\mathbf p = \mathbf a \circ \mathbf b$ for some $\mathbf a\in\mathcal A$ and $\mathbf b\in B$, then
\[
H(\mathbf p)\geq H_{\mathcal A}(\mathbf p) + H_{\mathcal B}(\mathbf p),
\]
with equality if and only if $\mathbf a$ and $\mathbf b$ achieve $H_{\mathcal A}(\mathbf p)$ and $H_{\mathcal B}(\mathbf p)$.
\item[(ii)] If $\mathcal B\subseteq \mathcal A^*$ then
\[
H(\mathbf p)\leq H_{\mathcal A}(\mathbf p) + H_{\mathcal B}(\mathbf p).
\]
with equality if and only if $\mathbf p = \mathbf a\circ \mathbf b$ for some $\mathbf a\in \mathcal A$ and $\mathbf b\in \mathcal B$.
\end{description}
\end{lemma}
\proof

(i) We have
\begin{eqnarray}
H(\mathbf p)& =& -\sum_i p_i\log a_ib_i\nonumber\\
&=& -\sum_ip_i\log a_i-\sum_ip_i\log b_i\nonumber\\
&\geq&H_{\mathcal A}(\mathbf p) + H_{\mathcal B}(\mathbf p).
\end{eqnarray}
We have equality if and only if $\mathbf a$ and $\mathbf b$ achieve $H_{\mathcal A}(\mathbf p)$ and $H_
{\mathcal B}(\mathbf p)$.\\
(ii) Let $\mathbf a\in\mathcal A$ and $\mathbf b\in \mathcal B$ achieve $H_{\mathcal A}(\mathbf p)$ and $H_{\mathcal B}(\mathbf p)$, respectively. Then the strict concavity of the $\log$ function and the relation $\mathbf b^T\mathbf a\leq 1$ imply
\[
H_{\mathcal A}(\mathbf p) + H_{\mathcal B}(\mathbf p) - H(\mathbf p) = -\sum_ip_i\log\frac{a_ib_i}{p_i} \geq -\log\sum_i a_ib_i\geq 0.
\]
Equality holds if and only if $a_ib_i=p_i$ whenever $p_i>0$. But then since
\[
1\geq\sum_i a_ib_i\geq\sum_ip_i=1,
\]
equality also holds for those indices with $p_i=0$.
\qed

The following theorem which was previously proved in I. Csisz\'{a}r and et.~al. \cite{Csis} characterizes a pair of generating convex corners.
\begin{theorem}\emph{(I. Csisz\'{a}r and et.~al.).}
For convex corners $\mathcal A,~\mathcal B\subseteq\mathbb R_+^k$ the following are equivalent:
\begin{description}
\item[(i)] $\mathcal A^*\subseteq\mathcal B$,
\item[(ii)] $\left(\mathcal A,\mathcal B\right)$ is a generating pair,
\item[(iii)] $H(\mathbf p)\geq H_{\mathcal A}(\mathbf p) + H_{\mathcal B}(\mathbf p)$ for every probability density $\mathbf p\in \mathbb R_+^k$.
\end{description}
\end{theorem}
\section{Additivity and Sub-additivity }
If $\mathbf{a}\in \mathbb{R}_{+}^k$ and  $\mathbf{b}\in \mathbb{R}_{+}^l$ then their \emph{Kronecker product}\index{product!Kronecker product} $\mathbf{a}\otimes\mathbf{b}\in\mathbb{R}_{+}^{kl}$ is defined by
\[
(\mathbf{a}\otimes\mathbf{b})_{ij} = a_i.b_j,~~~~i=1,\cdots,k,~j=1,\cdots,l.
\]
Note that if $\mathbf{p}$ and $\mathbf{q}$ are probability distributions then $\mathbf{p}\otimes\mathbf{q}$ is the usual product distribution. If $k=l$, then also the Schur product $\mathbf{a}\circ\mathbf{b}\in\mathbb{R}_{+}^k$ is defined by
\[
\mathbf{a}\circ\mathbf{b} = a_i.b_i,~~~~i=1,\cdots,k.
\]
Let $\mathcal{A}\subseteq \mathbb{R}_{+}^k$ and $\mathcal{B}\subseteq \mathbb{R}_{+}^l$ be convex corners. Their \emph{Kronecker product} $\mathcal{A}\otimes \mathcal{B}\subseteq \mathbb{R}_{+}^{kl}$ is the convex corner spanned by the Kronecker products $\mathbf{a}\otimes\mathbf{b}$ such that $\mathbf{a}\in \mathcal{A}$ and $\mathbf{b}\in \mathcal{B}$. The
\emph{Schur product} $\mathcal{A}\odot \mathcal{B}$ of the convex corners $\mathcal{A},~\mathcal{B}\subseteq \mathbb{R}_{+}^{k}$ is the convex corner in that same space spanned by the vectors $\mathbf{a}\circ\mathbf{b}$ such that $\mathbf{a}\in \mathcal{A}$ and $\mathbf{b}\in \mathcal{B}$. Thus
\[
\mathcal{A}\odot \mathcal{B} = \text{Convex Hull of }\left\{\mathbf{a}\circ\mathbf{b}:~\mathbf{a}\in \mathcal{A},~\mathbf{b}\in \mathcal{B}\right\}.
\]

 I. Csisz\'{a}r  et.~al.  proved the following lemma and theorem in \cite{Csis}.

\begin{lemma}\emph{( I. Csisz\'{a}r  et.~al.).}\label{lem:antiblock}
Let $\mathcal A,~\mathcal B\subseteq \mathbb{R}_+^k$ be convex corners. The pair $\left(\mathcal A, \mathcal B\right)$ is an antiblocking pair if and only if
\[
H(\mathbf p) = H_{\mathcal A}(\mathbf p) + H_{\mathcal B}(\mathbf p)
\]
for every probability distribution $\mathbf p\in\mathbb{R}_+^k$.
\end{lemma}

\begin{theorem}\emph{( I. Csisz\'{a}r  et.~al.).}\label{thm:additivity}
Let $\mathcal{A}\subseteq \mathbb{R}_{+}^k$ and $\mathcal{B}\subseteq \mathbb{R}_{+}^l$ be convex corners, and $\mathbf{p}\in \mathbb{R}_{+}^k,~\mathbf{q}\in \mathbb{R}_{+}^l$ probability distributions. Then, we have
\[
H_{\mathcal{A}\otimes \mathcal{B}}(\mathbf{p}\otimes\mathbf{q}) = H_\mathcal{A}(\mathbf{p}) + H_\mathcal{B}(\mathbf{q}) = H_{(\mathcal{A}^*\otimes \mathcal{B}^*)^*}(\mathbf{p}\otimes\mathbf{q}),
\]
Furthermore, for convex corners $\mathcal{A}, ~\mathcal{B}\subseteq \mathbb{R}_+^k$, and a probability distribution $\mathbf{p}\in \mathbb{R}_+^k$, we have
\[
H_{\mathcal{A}\odot \mathcal{B}}(\mathbf{p}) \leq H_\mathcal{A}(\mathbf{p}) + H_\mathcal{A}(\mathbf{p}).
\]
\end{theorem}
\proof
For $a\in\mathcal A$ and $b\in \mathcal B$, we have $a\otimes b\in \mathcal A \otimes \mathcal B$, which implies
\begin{eqnarray}
H_{\mathcal A\otimes \mathcal B}\left(\mathbf{p}\otimes\mathbf{q}\right)&\leq& -\sum_{i=1}^k\sum_{j=1}^lp_iq_j\log a_i b_j\nonumber\\
&=&-\sum_{i=1}^{k}p_i\log a_i -\sum_{j=1}^lq_j\log b_j.\nonumber
\end{eqnarray}
Hence $H_{\mathcal A\otimes\mathcal B}\left(\mathbf{p}\otimes\mathbf{q}\right)\leq H_{\mathcal A} (\mathbf{p})+ H_{\mathcal B}(\mathbf{q})$. By Lemma \ref{lem:antiblock},
\[
H\left(\mathbf p\otimes \mathbf q\right) = H_{\mathcal A\otimes \mathcal B}\left(\mathbf p\otimes \mathbf q\right) + H_{\left(\mathcal A\otimes \mathcal B\right)^*}\left(\mathbf p\otimes \mathbf q\right).
\]
Since $(\mathcal A)^*\otimes (\mathcal B)^*\subseteq \left(\mathcal A\otimes \mathcal B\right)^*$, we obtain
\begin{eqnarray}
H\left(\mathbf p\otimes\mathbf q\right)&\leq& H_{\mathcal A\otimes\mathcal B}\left(\mathbf p\otimes\mathbf q\right) + H_{\mathcal A^*\otimes\mathcal B^*}\left(\mathbf p\otimes\mathbf q\right)\label{eq:entp}\\
&\leq&H_{\mathcal A}(\mathbf p) + H_{\mathcal B}(\mathbf q) + H_{\mathcal A^*}(\mathbf p) + H_{\mathcal B^*}(\mathbf q)\nonumber\\
&\leq& H(\mathbf p) + H(\mathbf q)\nonumber\\
&=& H\left(\mathbf p\otimes\mathbf q\right).\nonumber
\end{eqnarray}
Thus we get equality everywhere in (\ref{eq:entp}), proving
\[
H_{\mathcal A\otimes \mathcal B}\left(\mathbf p\otimes\mathbf q\right) = H_{\mathcal A}(\mathbf p) + H_{\mathbf B}(\mathbf q),
\]
and consequently,
\[
H_{\left(\mathcal A\otimes \mathcal B\right)^*}\left(\mathbf p\otimes\mathbf q\right) = H_{\mathcal A^*\otimes \mathcal B^*}\left(\mathbf p\otimes\mathbf q\right) = H_{\mathcal A^*}(\mathbf p) + H_{\mathbf B^*}(\mathbf q).
\]
The second claim of the theorem is obviously true.
\qed

As an example let $G_1=\left(V_1,E_1\right)$ and $G_2=\left(V_2,E_2\right)$ be two graphs. The \emph{OR} product\index{product!OR product} of $G_1$ and $G_2$ is the graph $G_1\bigvee G_2$ with vertex set $V\left(G_1\bigvee G_2\right) = V_1\times V_2$ and $(v_1,v_2)$ is adjacnet to $(u_1,u_2)$ if and only if $v_1$ is adjacent to $u_1$ or $v_2$ is adjacent to $u_2$.  It follows that $VP\left(G_1\bigvee G_2\right) = VP\left(G_1\right)\otimes VP\left(G_2\right)$. From the above theorem we have
\[
H_k\left(G_1\bigvee G_2,\mathbf p\otimes \mathbf q\right) = H_k\left(G_1,\mathbf p\right) + H_k\left(G_2,\mathbf q\right).
\]
Thus if uniform probability densities on the vertices of $G_1$ and $G_2$ maximize  $H_k\left(G_1,\mathbf p\right)$ and  $H_k\left(G_2,\mathbf q\right)$ then the uniform probability density on the vertex of $G_1\bigvee G_2$ maximizes $H_k\left(G_1\bigvee G_2,\mathbf p\otimes \mathbf q\right)$.
\section{Perfect Graphs\index{perfect graphs} and Graph Entropy}

A graph $G$ is \emph{perfect} if for every induced subgraph $G^\prime$ of $G$, the chromatic number of $G^\prime$ equals the maximum size of a clique in $G^\prime$. Perfect graphs introduced by Berge in \cite{Brg} (see C. Berge \cite{Brg} and L. Lov\'{a}sz \cite{Lvzs}).

We defined the vertex packing polytope $VP(G)$ of a graph, in the previous sections. Here, we need another important notion from graph theory, i.e, the \emph{fractional vertex packing polytope} of a graph $G$. The \emph{fractional vertex packing polytope}\index{polytope!fractional vertex packing polytope} of G is defined as
\[
FVP(G) = \{\mathbf{b} \in \mathbb{R}^{|V|} : \mathbf{b}\geq 0,
\sum_{i \in K} b_i \leq 1~ \text{for all cliques $K$ of}~ G\}
\]
It is easy to see that, similar to $VP(G)$, the fractional vertex packing polytope $FVP(G)$ is also a convex corner and $VP(G)\subseteq FVP(G)$ for every graph $G$. Equality holds here if and only if the graph is perfect (See V. Chv\'{a}tal \cite{chv} and D. R. Fulkerson \cite{Flker}). Also note that
\[
FVP(G) = \left(VP(\overline{G})\right)^*.
\]
\begin{lemma}\emph{(I. Csisz\'{a}r and et.~al.).}\label{lem:lemma1}
Let $S = \{\mathbf{x}\geq 0 , \sum_i x_i \leq 1\}$. Then we have
\[
S = VP(G)\odot FVP(\overline{G}) = FVP(G) \odot VP (\overline{G}).
\]
Furthermore,
\[
VP(G)\odot VP (\overline{G}) \subseteq FVP(G)\odot FVP(\overline{G})
\]
\end{lemma}

A graph $G = (V,E)$ is \emph{strongly splitting} if for every probability distribution $P$ on $V$, we have
\[
H(P) = H_k\left(G,P\right) + H_k\left(\overline{G},P\right).
\]
K\"{o}rner and Marton in \cite{Jkor4} showed that bipartite graphs are strongly splitting while odd cycles are not.

Now, consider the following lemma which was previously proved in I. Csisz\'{a}r and et.~al. \cite{Csis}.
\begin{lemma}\label{lem:split1}
Let $G$ be a graph. For a probability density $P$ on $V(G)$, we have $H(P) = H_k\left(G,P\right) + H_k\left(\overline{G},P\right)$ if and only if $H_{VP(G)}(P) = H_{FVP(G)}(P)$.
\end{lemma}
\proof

We have $\left[VP(\overline{G})\right]^* = FVP(G)$. Thus, Lemma \ref{lem:equiv} and Lemma \ref{lem:antiblock} imply
\begin{eqnarray}
H_k\left(G,P\right) + H_k\left(\overline{G},P\right) - H(P) &=& H_{VP(G)}(P) + H_{VP(\overline{G})}(P) - H(P)\nonumber\\
&=& H_{VP(G)}(P) - H_{FVP(G)}(P).\nonumber
\end{eqnarray}
\qed

The following theorem conjectured by K\"{o}rner and Marton in \cite{Jkor4} first and proved by I. Csisz\'{a}r and et.~al. in \cite{Csis}.

\begin{theorem}\emph{(I. Csisz\'{a}r and et.~al.).}
A graph is strongly splitting if and only if it is perfect.
\end{theorem}
\proof

By Lemmas \ref{lem:corner1} and \ref{lem:split1}, $G$ is strongly splitting if and only if $VP(G) = FVP(G)$. This is equivalent to the perfectness of $G$.
\qed

Let $G = (V,E)$ be a graph with vertex set $V$ and edge set $E$. The graph
\[
G^{[n]} = (V^{n}, E^{[n]})
\]
is the \emph{$n$-th normal
power} where $V^n$ is the set of sequences of length $n$ from $V$, and two distinct vertices $x$ and $y$ are adjacent in $G^{[n]}$ if all of their entries are adjacent or equal in $G$, that is
\[
E^{[n]} = \{(x,y) \in V^n \times V^n :  x \neq y, ~\forall i~ (x_i,y_i) \in E~ \text{or}~ x_i = y_i\}.
\]
The \emph{$\pi$-entropy of a graph} $G=(V,E)$ with respect to the probability density $P$ on $V$ is defined as
\[
H_\pi(G,P) = \lim_{\epsilon \rightarrow 0} \lim_{n\rightarrow\infty} \min_{ U\subseteq V^n, \mathbf p^n(U)\geq 1 - \epsilon}
\frac{1}{n} \log\chi(G^{[n]}(U)).
\]
Note that $\overline{G^{[n]}} = \overline{G}^{(n)}$.

The follwoing theorem is proved in G. Simonyi \cite{Simu}.
\begin{theorem}\emph{(G. Simonyi).}
If $G=(V,E)$ is perfect, then $H_\pi(G,P)=H_k(G,P)$.
\end{theorem}

%\newpage\null\thispagestyle{empty}\newpage
\chapter{Chromatic Entropy}
In this chapter, we investigate minimum entropy colouring of the vertex set of a probabilistic graph $(G,P)$ which was previously studied by N. Alon and A. Orlitsky \cite{Alon96}. The minimum number of colours $\chi_H(G,P)$ required in a minimum entropy colouring of $V(G)$ was studied by J. Cardinal and et.~al. \cite{JC04} and \cite{JC08}. We state their results and further investigate $\chi_H(G,P)$.
\section{Minimum Entropy Coloring}

Let $X$ be a random variable distributed over a countable set $V$ and $\pi$ be a partition of $V$, i.e., $\pi = \{C_1,\cdots,C_k\}$ and $V = \cup_{i = 1}^kC_i$. Then $\pi$ induces a probability distribution on its cells, that is

\[
p(C_i) = \sum_{v\in C_i} p(v), \forall i\in\{1,\cdots,k\}.
\]

Therefore, the cells of $\pi$ have a well-defined entropy as follows:

\[
H\left(\pi\right) = \sum_{i = 1}^k p\left(C_i\right) \log \frac{1}{p\left(C_i\right)},
\]

If we consider $V$ as the vertex set of a probabilistic graph $(G,P)$ and $\pi$ as a partitioning of the vertices of $G$ into colour classes, then $H\left(\pi\right)$ is the entropy of a proper colouring of $V(G)$.

The \emph{chromatic entropy} of a probabilistic graph $(G,P)$ is defined as

\[
H_{\chi}(G,P) := \min\{H\left(\pi\right): \pi~\text{is a colouring of G}\},
\]

i.e. the lowest entropy of any colouring of $G$.

\begin{example}
We can colour the vertices of an empty graph with one colour. Thus, an empty graph has chromatic entropy 0. On the other hand, in a proper colouring of the vertices of a complete graph, we require distinct colours for distinct vertices. Hence, a complete graph has chromatic entropy $H(X)$. 

Now consider a 5-cycle with two different probability distributions over its vertices, i.e., uniform distribution and another one given by
$p_1 = 0.3$, $p_2 = p_3 = p_5 = 0.2$, and $p_4 = 0.1$. In both of them we require three colours. In the first one, a colour is assigned to a single vertex and each of the other two colours are assigned to two vertices. Therefore, the first probabilistic 5-cycle has chromatic entropy 
\[
H(0.4,0.4,0.2) \approxeq 1.52.
\]
For the second probabilistic 5-cycle, the chromatic entropy is attained by choosing the colour classes as $\{1,3\}$, $\{2, 5\}$, and $\{4\}$. Then, its chromatic entropy is 
\[
H(0.5,0.4,0.1) \approxeq 1.36.
\]
\end{example}

\section{Entropy Comparisons}

A \emph{source code} $\phi$ for a random variable $X$ is a mapping from the range of $X$, i.e., $\mathcal X$, to the set of finite-length strings, i.e., $\mathcal D^*$, of a $D$-ary alphabet. Let $\phi(x)$ denote the codeword corresponding to $x$ and let $l(x)$ be the length of $\phi(x)$. Then the average length $L(\phi)$ of the source code $\phi$ is
\[
L(\phi) = \sum_{x\in\mathcal X}p(x)l(x).
\]
The source coding problem is the problem of representing a random variable by a sequence of bits such that the expected length of the representation is minimized.

N. Alon and A. Orlitsky \cite{Alon96} considered a source coding problem in which a sender wants to transmit an information source to a receiver with some related data to the intended information source. Motivated by this problem, they considered the \emph{OR product} of graphs, as we stated in the previous chapter. We recall this graph product here. 

Let $G_1,\cdots,G_n$ be graphs with vertex sets $V_1,\cdots,V_n$. The \emph{OR product}\index{product!OR product} of $G_1,\cdots,G_n$ is the graph $\bigvee_{i=1}^nG_i$ whose vertex set is $V^n$ and where two distinct vertices $(v_1,\cdots,v_n)$ and $(v^\prime_1,\cdots,v^\prime_n)$ are adjacent if for some $i\in\{1,\cdots,n\}$ such that $v_i\neq v^\prime_i$, $v_i$ is adjacent to $v^\prime_i$ in $G_i$. The $n$-fold OR product of $G$ with itself is denoted by $G^{\bigvee n}$.

N. Alon and A. Orlitsky \cite{Alon96} proved the following lemma which relates chromatic entropy to graph entropy.
\begin{lemma}\emph{(N. Alon and A. Orlitsky).}\label{lem:H0}
$\lim_{n\rightarrow\infty} \frac{1}{n} H_{\chi} (G^{\bigvee n}, P^{(n)}) = H_k(G,P)$.
\end{lemma}

Let $\Omega(G)$ be the collection of cliques of a graph $G$. The \emph{clique entropy} of a probabilistic graph $(G,P)$ is
\[
H_{\omega}(G,P):=\max\{H(X|Z^\prime):X\in Z^\prime\in\Omega(G)\}.
\]
That is, for every vertex $x$ we choose a conditional probability distribution $p(z^\prime|x)$ ranging over the cliques containing $x$. This determines a joint probability distribution of $X$ and a random variable $Z^\prime$ ranging over all cliques containing $X$. Then, the clique entropy is the maximal conditional entropy of $X$ given $Z^\prime$.

\begin{example}
The only cliques of an empty graph are singletones. Thus for an empty graph, we have
\[
Z^\prime = \{X\},
\]
which implies 
\[
H_\omega(G,P) = 0.
\]
On the other hand, for a complete graph, we can take $Z^\prime$ to be the set of all vertices. Thus, for a probabilistic complete graph $(G,P)$, we have
\[
H_\omega(G,P) = H(X).
\]
For a 5-cycle, every clique is either a singleton or an edge. Thus, for a probabilistic 5-cycle with uniform distribution over the vertices, we have
\[
H_\omega(G,P)\leq 1.
\] 
Now, if for every $x$ we let $Z^\prime$ be uniformly distributed over the two edges containing $x$, then by symmetry we get
\[H(X|Z^\prime) = 1,
\]
which implies 
\[
H_\omega(G,P) = 1.
\]
\end{example}

N. Alon and J. Cardinal and et.~al. proved the following lemmas in \cite{Alon96} and \cite{JC08}.

\begin{lemma}\emph{(N. Alon and A. Orlitsky).}\label{lem:H1}
Let $U$ be the uniform distribution over the vertices $V(G)$ of a probabilistic graph $(G,U)$ and $\alpha(G)$ be the independence number of the graph $G$. Then,
\[
H_{\chi}(G,U)\geq \log\frac{|V(G)|}{\alpha(G)}.
\]
\end{lemma}
\qed

\begin{lemma}\emph{(N. ALon and A. Orlitsky).}
For every probabilistic graph $(G,P)$
\[
H_\omega(G,P) = H(P) - H_k(\overline{G},P).
\]
\end{lemma}
\qed

\begin{lemma}\emph{(J. Cardinal and et.~al.).}\label{lem:H2}
For every probabilistic graph $(G,P)$, we have
\[
-\log\alpha(G,P)\leq H_k(G,P) \leq H_{\chi}(G,P) \leq \log\chi(G).
\]
Here $\alpha(G,P)$ denotes the maximum weight $P(S)$ of an independent set $S$ of $(G,P)$.
\end{lemma}
\qed

It may seem that non-uniform distribution decreases chromatic entropy $H_\chi(G,P)$, but the following example shows that this is not true.
Let us consider 7-star with $deg(v_1)=7$ and $deg(v_i) =1$ for $i\in\{2,\cdots,8\}$. If $p(v_1) = 0.5$ and $p(v_i) = \frac{1}{14}$ for $i\in\{2,\cdots,8\}$, then $H_\chi(G,P) = H(0.5,0.5) = 1$, while if $p(v_i) = \frac{1}{8}$ for $i\in\{1,\cdots,8\}$, then $H_\chi(G,P) = H(\frac{1}{8},\frac{7}{8})\leq H(0.5,0.5) = 1$.

\section{Number of Colours and Brooks' Theorem}
Here, we investigate  the minimum number of colours $\chi_{H}(G,P)$ in a minimum entropy colouring of a probabilistic graph $(G,P)$. First, we have the following definition.

A \emph{Grundy colouring}\index{grundy colouring} of a graph is a colouring such that for any colour $i$, if a vertex has colour $i$ then it is adjacent to at least one vertex of colour $j$ for all $j<i$. The Grundy number $\Gamma(G)$ of a graph $G$ is the maximum number of colours in a Grundy colouring of $G$. Grundy colourings are colourings that can be obtained by iteratively removing maximal independent sets.

The following theorem was proved in J. Cardinal and et.~al. \cite{JC08}.

\begin{theorem}\emph{(J. Cardinal and et.~al.)}\label{thm:Grundy}
Any minimum entropy colouring of a graph $G$ equipped with a probability distribution on its vertices is a Grundy colouring. Moreover, for any Grundy colourig $\phi$ of $G$, there exists a probability mass function $P$ over $V(G)$ such that $\phi$ is the unique minimum entropy colouring of $(G,P)$.
\end{theorem}

We now consider upper bounds on $\chi_H(G,P)$ in terms of the maximum valency of $G$, i.e., $\Delta(G)$. The following theorems were proved in J. Cardinal and et.~al. \cite{JC04} and J. Cardinal and et.~al. \cite{JC08}.

\begin{theorem}\emph{(J. Cardinal and et.~al.).}
For any probabilistic graph $(G,P)$, we have $\chi_H(G,P)\leq\Delta(G)+1$.
\end{theorem}

\begin{theorem}\emph{(Brooks' Theorem for Probabilistic Graphs).}
If $G$ is connected graph different from a complete graph or an odd cycle, and $U$ is a uniform distribution on its vertices, then $\chi_H(G,U)\leq\Delta(G)$.
\end{theorem}

\section{Grundy Colouring and Minimum Entropy Colouring}
Let $\phi: v_1,v_2,\cdots,v_n$ be an ordering of the vertices of a graph $G$. A proper vertex colouring $c:V(G)\rightarrow \mathbb{N}$ of $G$ is a $\phi-$colouring of $G$ if the vertices of $G$ are coloured in the order $\phi$, beginning with $c(v_1) = 1$, such that each vertex $v_{i+1} (1\leq i \leq n-1)$ must be assigned a colour that has been used to colour one or more of the vertices $v_1,v_2,\cdots,v_i$ if possible. If $v_{i+1}$ can be assigned more than one colour, then a colour must be chosen which results in using the fewest number of colours needed to colour $G$. If $v_{i+1}$ is adjacent to vertices of every currently used colour, then $c(v_{i+1})$ is defined as the smallest positive integer not yet used. The parsimonious $\phi-$colouring number $\chi_{\phi}(G)$ of $G$ is the minimum number of colours in a $\phi-$colouring of $G$. The maximum value of  $\chi_{\phi}(G)$ over all orderings $\phi$ of the vertices of $G$ is the oredered chromatic number or, more simply, the ochromatic number of $G$, which is denoted by $\chi^{o}(G)$.

Paul Erd\"{o}s, William Hare, Stephen Hedetniemi, and Renu Lasker proved the following lemma in \cite{Erdos03}.

\begin{lemma}\emph{(Erd\"{o}s and et.~al.).}
For every graph $G$, $\Gamma(G) = \chi^o(G)$.
\end{lemma}

Now we prove the following lemma.

\begin{lemma}
For every probabilistic graph $(G,P)$, we have
\[
\max_P \chi_H(G,P) = \Gamma(G).
\]
\end{lemma}

\begin{proof}
Due to Theorem \ref{thm:Grundy} any minimum entropy colouring of a graph $G$ equipped with a probability distribution on its vertices is a Grundy colouring, and for any Grundy colouring $\phi$ of $G$, there exists a probability distribution $P$ over $V(G)$ such that $\phi$ is the unique minimum entropy colouring of $(G,P)$.
\end{proof}

\begin{corollary}
\[
\max_P \chi_H(G,P) = \chi^o(G,P).
\]
\end{corollary}

Note that every greedy colouring of the vertices of a graph is a Grundy colouring.

It is worth mentioning that the chromatic number of a vertex transitive graph is not achieved by a Grundy colouring. Let $G$ be a 6-cycle. Consider the first Grundy colouring of $G$ with colour classes $\{v_1,v_4\}$, $\{v_3,v_6\}$, and $\{v_2,v_5\}$ and the second Grundy colouring with colour classes $\{v_1,v_3,v_5\}$ and $\{v_2,v_4,v_6\}$.

It may seem that every Grundy colouring of a probabilistic graph is a minimum entropy colouring, but the following example shows that is not true.
Consider a probability distribution for the 6-cycle in the above example as $p(v_1) = p(v_3) = 0.4$ and $p(v_2) = p(v_4) = p(v_5) = p(v_6) = 0.05$. Then, denoting the first Grundy colouring in the example above by $c_A$ and the second one by $c_B$, we have $H(c_A) = 0.44$ and $H(c_B) = 0.25$ which are not equal.

\begin{remark}\label{rem:cnj1}
Let $(G,P)$ and $(G^\prime,P^\prime)$ be two probabilistic graphs, and $\phi:G\rightarrow G^\prime$ a homomorphism from $G$ to $G^\prime$ such that for every $v^\prime \in V(G^\prime)$, we have
\[
p^\prime(v^\prime) = \sum_{v:v \in \phi^{-1}(v^\prime)}p(v).
\]
Then, can we say that
\[
\chi_H(G,P)\leq\chi_H(G^\prime,P^\prime)?
\]
The following example shows that is not true. Let $(G,P)$ be a probabilistic 6-cycle and $(G^\prime,P^\prime)$ be a probabilistic $K_2$ with the corresponding probability distributions as follows. $p(v_1) = p(v_4) = 0.4$, $p(v_2) = p(v_3) = p(v_5) = p(v_6) = 0.05$. Then, $\chi_H(C_6,P) = 3$ while $\chi_H(K_2,P^\prime) = 2$, i.e., $\chi_H(C_6,P)\geq\chi_H(K_2,P^\prime)$. It is worth noting that even as a result of simple operations like deleting an edge, we cannot have the above conjecture. To see this just add an edge between $v_1$ and $v_4$ in this example.
\end{remark}

\section{Minimum Entropy Colouring and Kneser Graphs}

In this section, we study the minimum entropy colouring of a Kneser graph $K_{v:r}$ and prove the following Theorem.

\begin{theorem}\label{thm:Kneser}
Let $(K_{v:r}, U)$ be a probabilistic Kneser graph with uniform distribution $U$ over its vertices and $v\geq 2r$. Then, the minimum number of colours in a minimum entropy colouring of $(K_{v:r}, U)$, - i.e. $\chi_H(K_{v:r}, U)$, is equal to the chromatic number of $K_{v:r}$, i.e. $\chi(K_{v:r})$. Furthermore, the chromatic entropy of  $(K_{v:r}, U)$ is
\begin{multline}
H_{\chi}(K_{v:r}, U)= \frac{1}{\chi_f(K_{v:r})} \log\chi_f(K_{v:r}) +\\
\sum_{0\leq i \leq v-1-2r} \frac{1}{\chi_f(K_{v:r})} \frac{1}{\prod_{j = 0}^i\chi_f(K_{v-j-1:v-r-j})}\log\chi_f(K_{v:r})\prod_{j = 0}^i\chi_f(K_{v-j-1:v-r-j}).\nonumber
\end{multline}

\end{theorem}

Before proving the above theorem, we explain some preliminaries and a lemma which were previously given in J. Cardinal and et.~al. \cite{JC08}.

Consider a probabilistic graph $(G,P)$. Let $S$ be a subset of the vertices of $G$, i.e., 
\[
S \subseteq V(G).
\]
Then $P(S)$ denotes
\[
P(S):=\sum_{x\in S}p(x).
\] 
Note that a \emph{colouring} of $V(G)$ is a map $\phi$ from the vertex set $V(G)$ of $G$ to the set of positive integers $\mathbb{N}$, that is
\[
\phi:V(G)\rightarrow \mathbb{N}.
\]
Then $\phi^{-1}(i)$ denotes the set of vertices coloured with colour $i$. Let $c_i$ be the probability of the $i$-th colour class. Hence, letting $X$ be a random vertex with distribution $P$ ranging over the vertices of $G$, we get 
\[
c_i = P(\phi^{-1}(i))= P(\phi(X)=i). 
\]
The \emph{colour sequence} of $\phi$ with respect to $P$ is the infinite vector $c=(c_i)$.

Let  $(G,P)$ be a probabilistic graph. A sequence $c$ is said to be \emph{colour-feasible} if there exists a colouring $\phi$ of $V(G)$ having $c$ as colour sequence. We consider nonincreasing colour sequences,  that is, colour sequences $c$ such that 
\[
c_i\geq c_{i+1},~\forall~i.
\]
Note that colour sequences define discrete probability distributions on $\mathbb{N}$. Then the \emph{entropy of colour sequence c} of a colouring $\phi$, i.e., $H(c)$ is 
\[
H(c) = -\sum_{i\in\mathbb{N}} c_i \log c_i.
\]

The following lemma was proved in N. Alon and A. Orlitsky\cite{Alon96}.

\begin{lemma}\emph{(N. Alon and A. Orlitsky).}
Let $c$ be a nonincreasing colour sequence, let $i$, $j$ be two indices such that $i<j$ and let $\alpha$ a real number such that $0<\alpha\leq c_j$. Then we have $H(c)>H(c_1,\cdots,c_{i-1},c_i + \alpha, c_{i+1}, \cdots, c_{j-1}, c_j - \alpha, c_{j+1},\cdots)$.
\end{lemma}
 We now examine the consequences of this lemma. We say that a colour sequence $c$ \emph{dominates} another colour sequence $d$ if $\sum_{i=1}^jc_j\geq\sum_{i=1}^jd_i$ holds for all $j$. We denote this by $c\succeq d$.Note that $\succeq$ is a partial order. We also let $\succ$ denote the strict part of $\succeq$. The next lemma which was proved in J. Cardinal and et.~al. \cite{JC08} shows that colour sequences of minimum entropy colourings are always maximal colour feasible.

\begin{lemma}\emph{(J. Cardinal and et.~al.).} \label{lem:colourseq}
Let $c$ and $d$ be two nonincreasing rational colour sequences such that $c\succ d$. Then we have $H(c)<H(d)$.
\end{lemma}

Now we prove Theorem \ref{thm:Kneser}.

\begin{proofthm}
The proof is based on induction on $v$. For $v=2r$ the assertion holds. We prove the assertion for $v>2r$. Due to Erd\"{o}s-Ko-Rado theorem, the colour sequence corresponding to the grundy colouring acheiveing the chromatic number of a Kneser graph dominates all colour feasble sequences. Hence using Lemma \ref{lem:colourseq}, we have $\chi_H(K_{v-1:r},U) = \chi(K_{v-1:r})$. Now, removing the maximum size coclique in $K_{v:r}$, due to induction hypothesis, we get a minimum entropy colouring of $K_{v-1:r}$. Thus we have $\chi_H(K_{v:r},U) - 1 = \chi_H(K_{v-1:r},U)= \chi(K_{v-1:r})$. Noting that $\chi(K_{v:r}) = \chi(K_{v-1:r}) + 1$, we have $\chi_H(K_{v:r},U) = \chi(K_{v:r})$.
\end{proofthm}

\begin{corollary}
Let $G_1= \left(K_{v:r},U\right)$, and $\left(G_2,U\right)$ is homomorphically equivalent, in the sense of Remark \ref{rem:cnj1}, to $G_1$, then we have
$\chi_H\left(G_1\right) = \chi_H\left(G_2\right) = \chi\left(G_1\right) = \chi\left(G_2\right)$.
\end{corollary}
\section{Further Results}

As we mentioned in the previous chapter,  for a probabilistic graph $(G,P)$, we have

\begin{equation}
\max_{P} H_k\left(G,P\right) = \log \chi_f\left(G\right). \label{H}
\end{equation}

In this section, we prove the following theorem for vertex transitive graphs using chromatic entropy. Recall that we gave another proof of the following theorem using the structure of vertex transitive graphs and convex optimization techniques in previous chapter.

\begin{theorem}\label{thm:vxtrans}
Let $G$ be a vertex transitive graph. Then the uniform distribution over vertices of $G$ maximizes $H_k\left(G,P\right)$.
That is $H_k\left(G,U\right) = \log \chi_f\left(G\right)$.
\end{theorem}

\begin{proof}
First note that for a vertex transitive graph $G$, we have $\chi_f\left(G\right) = \frac{|V(G)|}{\alpha(G)}$, and the $n$-fold OR product
$G^{\bigvee n}$ of a vertex transitive graph $G$ is also vertex transitive.  Now from Lemma \ref{lem:H1}, Lemma \ref{lem:H2} , and equation \ref{H},
we have

\begin{equation}
H_k\left(G^{\bigvee n},U\right) \leq \log \chi_f\left(G^{\bigvee n}\right) \leq H_{\chi}\left(G^{\bigvee n},U\right),\label{eq:U1}
\end{equation}

From \cite{Alon96} and \cite{fgt}, we have $H_k\left(G^{\bigvee n},U\right) = n H_k\left(G,U\right)$,
$\chi_f\left(G^{\bigvee n}\right) =  \chi_f\left(G\right)^n $, and $\log \chi_f\left(G\right)  = \lim_{n\rightarrow\infty} \frac{1}{n}\log\chi\left(G^{\bigvee n}\right)$. Hence, applying Lemma \ref{lem:H0} to equation \ref{eq:U1} and using squeezing theorem, we get

\begin{equation}
H_k\left(G,U\right) = \log \chi_f\left(G\right) = \lim_{n\rightarrow\infty} \frac{1}{n}\log\chi\left(G^{\bigvee n}\right) =
\lim_{n\rightarrow\infty} \frac{1}{n}H_\chi\left(G^{\bigvee n},U\right).\label{U2}
\end{equation}
\end{proof}

The following example shows that the converse of the above theorem is not true. Consider $G = C_4 \cup C_6$, with vertex sets $V(C_4) = \{v_1,v_2,v_3,v_4\}$ and $V(C_6) = \{v_5,v_6,v_7,v_8,v_9,v_{10}\}$, and parts $A = \{v_1,v_3,v_5,v_7,v_9\}$, $B = \{v_2,v_4,v_6,v_8,v_{10}\}$. Clearly, $G$ is not a vertex transitive graph, however, using Theorem \ref{thm:bipentropy}, one can see that the uniform distribution $U = \left(\frac{1}{10},\cdots,\frac{1}{10}\right)$ gives the maximum graph entropy which is $1$.

\begin{remark}
Note that the maximizer probability distribution of the graph entropy is not unique. Consider $C_4$ with vertex set $V(C_4) = \{v_1,v_2,v_3,v_4\}$ with parts $A = \{v_1,v_3\}$ and $B = \{v_2,v_4\}$. Using Theorem \ref{thm:bipentropy}, probability distributions $P_1 = (\frac{1}{4},\frac{1}{4},\frac{1}{4},\frac{1}{4})$ and $P_2 = (\frac{1}{8},\frac{1}{4},\frac{3}{8},\frac{1}{4})$ give the maximum graph entropy which is $1$.
\end{remark}

Now note that we can describe the chromatic entropy of a graph in terms of the graph entropy of a complete graph as
\[
H_\chi\left(G,P\right) = \min\{H_k\left(K_n,P^\prime\right): \left(G,P\right)\rightarrow \left(K_n,P^\prime\right) \}.
\]
A graph $G$ is called \emph{symmetric with respect to a functional $F_G\left(P\right)$} defined on the set of all the probability distributions on its vertex set if the distribution $P^*$ maximizing $F_G\left(P\right)$ is uniform on $V\left(G\right)$. We study this concept in more detail in the next chapter.

%\newpage\null\thispagestyle{empty}\newpage
\chapter{Symmetric Graphs\index{symmetric graphs}}

A graph $G$ with distribution $P$ on its vertices is called \emph{symmetric with respect to graph entropy} $H_k\left(G,P\right)$ if the uniform probability distribution on its vertices maximizes $H_k\left(G,P\right)$. In this chapter we characterize different classes of graphs which are symmetric with respect to graph entropy.
\section{Symmetric Bipartite Graphs}

\begin{theorem}\label{thm:Bip}
Let $G$ be a bipartite graph with parts $A$ and $B$, and no isolated vertices. The uniform probability distribution $U$ over the vertices of $G$ maximizes $H_k\left(G,P\right)$ if and only if $G$ has a perfect matching.
\end{theorem}

\proof

Suppose $G$ has a perfect matching, then $|A| = |B|$, and due to Hall's theorem we have

\[
|D|\leq |\mathcal{N}(D)|,~~~\forall D\subseteq A.
\]

Now assuming $P = U$, we have

\[
p(D) = \frac{|D|}{|V(G)|},~~~~p(A) = \frac{|A|}{|V(G)|} = \frac{|B|}{|V(G)|} = p(B),
\]

Thus, the condition of Theorem \ref{thm:bipentropy} is satisfied, that is

\[
\frac{p(D)}{p(A)}\leq \frac{p(\mathcal{N}(D))}{p(B)},~~~\forall D\subseteq A,
\]

Then, due to Theorem \ref{thm:bipentropy}, we have

\[
H_k\left(G,U\right) = h\left(p(A)\right) = h\left(\frac{1}{2}\right) = 1.
\]

Noting that $H_k\left(G,P\right)\leq \log \mathcal{X}_f(G),~\forall P$, and $\log \mathcal{X}_f(G) = 1$ for a bipartite graph $G$, the assertion holds.

Now suppose that $G$ has no perfect matching, then we show that $H_k\left(G,U\right) < 1$. First, note that from K\"onig's theorem we can say that a bipartite graph $G = (V,E)$ has a perfect matching if and only if each vertex cover has size at least $\frac{1}{2}|V|$. This implies that if a bipartite graph $G$ does not have a perfect matching, then $G$ has a stable set with size $> \frac{|V|}{2}$.

Furtthermore, as mentioned in \cite{Schriv1}, the stable set polytope of a graph $G$ is determined by the following inequalities if and only if $G$ is bipartite.

\begin{eqnarray}
&0\leq x_v \leq 1,~~~\forall v\in V(G),\nonumber\\
&x_u + x_v \leq 1, ~~~\forall e=uv\in E(G).\nonumber
\end{eqnarray}

We show $\max_{x\in~\text{stable set polytope}} \prod_{v\in V}x_v > 2^{-|V|}$. Let $S$ denote a stable set in $G$ with $|S|>\frac{|V|}{2}$. We define a vector $\mathbf{\overline{x}}$ such that $\overline{x}_v = \frac{|S|}{|V|}$ if $v\in S$ and $\overline{x}_v =1 -  \frac{|S|}{|V|}$ otherwise.
Since $|S|>\frac{|V|}{2}$, $\mathbf{\overline{x}}$ is feasible. Letting $t:= \frac{|S|}{|V|}$, we have

\begin{eqnarray}
&-\left(t \log t + (1 -t) \log (1 - t)\right) < 1,\nonumber\\
&\rightarrow \log t^{t} (1 - t)^{(1 - t)} > -1,\nonumber\\
&\rightarrow t^{t} (1 - t)^{(1 - t)} > 2^{-1},\nonumber\\
&\rightarrow \prod_{v\in V} \overline{x}_v > 2^{-|V|}.\nonumber
\end{eqnarray}
\qed

\section{Symmetric Perfect Graphs}

Let $G=\left(V,E\right)$ be a graph. Recall that the fractional vertex packing polytope of $G$,i.e, $FVP(G)$ is defined as
\[
FVP(G) := \{\mathbf{x}\in\mathbb{R}_+^{|V|}: \sum_{v\in K}x_v\leq 1~\text{for all cliques K of G}\}.
\]
Note that $FVP(G)$ is a convex corner and for every graph $G$, $VP(G)\subseteq FVP(G)$. The following theorem was previously proved in \cite{chv} and \cite{Flker}.

\begin{theorem}
A graph $G$ is perfect if and only if $VP(G) = FVP(G)$.
\end{theorem}

The following theorem which is called \emph{weak perfect graph theorem} is useful in the following discussion. This theorem was proved by  Lov\'{a}sz in \cite{Lov1} and \cite{Lov2} and is follows.
\begin{theorem}\label{thm:weakpf}
A graph $G$ is perfect if and only if its complement is perfect.
\end{theorem}
Now, we prove the following theorem which is a generalization of our bipartite symmetric graphs with respect to graph entropy.
\begin{theorem}\label{thm:sympf}
Let $G=(V,E)$ be a perfect graph and $P$ be a probability distribution on $V(G)$. Then $G$ is symmetric with respect to graph entropy $H_k\left(G,P\right)$ if and only if $G$ can be covered by its cliques of maximum size.
\end{theorem}

\proof

Suppose $G$ is covered by its maximum-sized cliques, say $Q_1,\cdots,Q_m$. That is $V(G) = V(Q_1) \dot{\cup}\cdots \dot{\cup}V(Q_m)$ and
$|V(Q_i)|=\omega(G),~\forall i\in [m]$.

Now, consider graph $T$ which is the disjoint union of the subgraphs induced by $V(Q_i)~\forall i\in [m]$. That $T = \dot{\bigcup}_{i=1}^m G\left[V(Q_i)\right]$. Noting that $T$ is a disconnected graph with $m$ components, using Corollary \ref{cor:EntrDisc} we have

\[
H_k\left(T,P\right) = \sum_i P(Q_i)H_k(Q_i,P_i).
\]

Now, having $V(T) = V(G)$ and $E(T)\subseteq E(G)$, we get $H_k\left(T,P\right)\leq H_k\left(G,P\right)$ for every distribution $P$. Using Lemma \ref{lem:keylemma}, this implies

\begin{equation}
H_k\left(T,P\right) = \sum_i P(Q_i)H_k\left(Q_i,P_i\right) \leq \log\chi_f(G),~\forall P,\label{eq:comps}
\end{equation}

Noting that $G$ is a perfect graph, the fact that complete graphs are symmetric with respect to graph entropy, $\chi_f\left(Q_i\right) = \chi_f(G) = \omega(G)=\chi(G),~\forall i\in [m]$, and \ref{eq:comps}, we conclude that uniform distribution maximizes $H_k\left(G,P\right)$.

Now, suppose that $G$ is symmetric with respect to graph entropy. We prove that $G$ can be covered by its maximum-sized cliques. Suppose this is not true. We show that $G$ is not symmetric with respect to $H_k\left(G,P\right)$.

Denoting the minimum clique cover number of $G$ by $\gamma(G)$ and the maximum independent set number of $G$ by $\alpha(G)$, from perfection of $G$ and weak perfect theorem, we get $\gamma(G) = \alpha(G)$. Then, using this fact, our assumption implies that $G$ has an independent set $S$ with $|S|> \frac{|V(G)|}{\omega(G)}$.

We define a vector $\overline{\mathbf{x}}$ such that $\overline{x}_v = \frac{|S|}{|V|}$ if $v\in S$ and $\overline{x}_v = \frac{1 - \frac{|S|}{|V|}}{\omega - 1}$ if $v\in V(G) \backslash S$. Then, we can see that $\overline{\mathbf{x}}\in FVP(G) = VP(G)$. Let $t:=\frac{|S|}{|V|}$. Then, noting that $t>\frac{1}{\omega}$,

\begin{eqnarray}
H_k\left(G,U\right)&\leq& -\frac{1}{|V|}\sum_{v\in V} \log \overline{x}_v\nonumber\\
&=& -\frac{1}{|V|}\left(\sum_{v\in S}\log\overline{x}_v + \sum_{v\in V\backslash S}\overline{x}_v\right)\nonumber\\
&=&-\frac{1}{|V|}\left(|S|\log\alpha + (|V| - |S|)\log\frac{1-\alpha}{\omega-1}\right)\nonumber\\
&=&-t\log t - (1 - t)\log\frac{1 - t}{\omega - 1}\nonumber\\
&=&-t\log t - (\omega -1)\left(\frac{1-t}{\omega-1}\log\frac{1 - t}{\omega -1}\right) < \log \omega(G).\nonumber
\end{eqnarray}
\qed

Note that we have
\[
\gamma\left(G\right) = \alpha\left(G\right).
\]
Now, considering that finding the clique number of a perfect graph can be done in polynomial time and using weak perfect graph theorem we conclude that one can decide in polynomial time whether a perferct graph is symmetric with respect to graph entropy.

\begin{corollary}
Let $G$ be a connected regular line graph without any isolated vertices with valency $k>3$. Then if $G$ is covered by its disjoint maximum-size cliques, then $G$ is symmetric with respect to $H_k(G,P)$.
\end{corollary}

\proof

Let $G = L(H)$ for some graph $H$. Then either $H$ is bipartite or regular. If $H$ is bipartite, then $G$ is perfect (See \cite{West}) and because of Theorem \ref{thm:sympf} we are done. So suppose that $H$ is not bipartite. Then each clique of size $k$ in $G$ corresponds to a vertex $v$ in $V(H)$ and the edges incident to $v$ in $H$ and vice versa. That is because any such cliques in $G$ contains a triangle and there is only one way extending that triangle to the whole clique which corresponds to edges incident with the corresponding vertex in $H$. This implies that the covering cliques in $G$ give an independent set in $H$ which is also a vertex cover in $H$. Hence $H$ is a bipartite graph and hence $G $ is perfect. Then due to Theorem \ref{thm:sympf} the theorem is proved.
\qed

\section{Symmetric Line Graphs}
In this section we introduces a class of line graphs which are symmetric with respect to graph entropy. Let $G_2$ be a line graph of some graph $G_1$, i.e, $G_2 = L(G_1)$.  Let $|V(G_1)|=n$ and $|E(G_1)| = m$. We recall that a vector $\mathbf x\in \mathbb R_+^m$ is in the matching polytope $MP(G_1)$ of the graph $G_1$ if and only if it satisfies (see A. Schrijver \cite{Schriv1}).
\begin{eqnarray}\label{eq:MatchPoly}
&x_e\geq 0~~~~~~~~ &\forall e\in E(G_1),\nonumber\\
&x(\delta(v))\leq 1 ~~~~~~~~~&\forall v\in V(G_1),\\
&x\left(E[U]\right)\leq\lfloor\frac{1}{2}|U|\rfloor, ~~~~~~~~~&\forall U\subseteq V(G_1)~\text{with}~|U|~\text{odd.}\nonumber
\end{eqnarray}
Let $\mathcal M$ denote the family of all matchings in $G_1$, and for every matching $M\in\mathcal M$ let the charactersitic vector $\mathbf b_M\in\mathbb R_+^m$ be as
\begin{equation}
(\mathbf b_M)_e = \left\{ \begin{array}{rcl}
1, & & e\in M ,\\
0, && e\notin M.
\end{array}\right.
\end{equation}
Then the \emph{fractional edge-colouring number}\index{fractional edge-colouring number} $\chi_f^\prime(G_1)$ of $G_1$ is defined as
\[
\chi_f^\prime(G_1):=\min\{\sum_{M\in\mathcal M}\lambda_M|\mathbf{\lambda}\in\mathbb R_+^{\mathcal M},~\sum_{M\in\mathcal M}\lambda_M\mathbf b_M = \mathbf 1\}.
\]
If we restrict $\lambda_M$ to be an integer, then the above definition give rise to the edge colouring number of $G_1$, i.e., $\chi^\prime(G_1)$. Thus
\[
\chi_f^\prime(G)\leq \chi^\prime(G).
\]
As an example considering $G_1$ to be the peterson graph, we have
\[
\chi_f^\prime(G)=\chi^\prime(G) = 3.
\]
\begin{remark} \label{rem:Rem2}
Note that every matching in $G_1$ corresponds to an independent set in $G_2$ and every independent set in $G_2$ corresponds to a matching in $G_1$. Note that the fractional edge-colouring number of $G_1$, i.e., $\chi_f^\prime(G_1)$ is equal to the fractional chromatic number of $G_2$, i.e.,$\chi_(G_2)$. Thus
\[
\chi_f^\prime(G_1) = \chi_f(G_2).
\]
Furthermore, note that the vertex packing polytope $VP(G_2)$ of $G_2$ is the matching polytope $MP(G_1)$ of $G_1$ (see L. Lov\'{a}sz and M. D. Plummer \cite{Lov3}). That is
\[
VP(G_2) = MP(G_1).
\]
\end{remark}
The following theorem which was proved by Edmond, gives a characterization of the fractional edge-colouring number $\chi_f^\prime(G_1)$ of a graph $G_1$ (see A. Schrijver \cite{Schriv1}).
\begin{theorem}\label{thm:edgecol1}
Let $\Delta(G_1)$ denote the maximum degree of $G_1$. Then the fractional edge-colouring number of $G_1$ is obtained as
\[
\chi_f^\prime(G_1) = \max\{\Delta(G_1),\max_{U\subseteq V,~|U|\geq3}\frac{|E(U)|}{\lfloor\frac{1}{2}|U|\rfloor}\}.
\]
\end{theorem}
Following A. Schrijver \cite{Schriv1} we call a graph $G_1$ a \emph{$k$-graph}\index{$k$-graph} if it is $k$-regular and its fractional edge coloring number $\chi_f^\prime(H)$ is equal to $k$.
The following colloray characterizes a $k$-graph (see Alexander Schrijver \cite{Schriv1}).
\begin{corollary}\label{cor:kgraph}
Let $G_1=(V_1,E_1)$ be a $k$-regular graph. Then $\chi^\prime_f(G_1)=k$ if and only if $|\delta(U)|\geq k$ for each odd subset $U$ of $V_1$.
\end{corollary}
The following theorem introduces a class of symmetric line graphs with respect to graph entropy. The main tool in the proof of the following theorem is Karush-Kuhn-Tucker (KKT)\index{Karush-Kuhn-Tucker (KKT)} optimality conditions in convex optimization (see S. Boyd and L. Vanderberghe \cite{Boyd}).
\begin{theorem}
Let $G_1$ be a $k$-graph with $k\geq3$. Then the line graph $G_2 = L(G_1)$ is symmetric with respect to graph entropy.
\end{theorem}
\proof

From our discussion in Remark \ref{rem:Rem2} above we have
\[
H_k\left(G_2,P\right) = \min_{\mathbf x\in MP(G_1) }\sum_{e\in E(G_1)} p_e\log \frac{1}{x_e},
\]
Let $\lambda_v,~\gamma_U\geq 0$ be the Lagrange multipliers corresponding to inequalities $x(\delta(v))\leq 1$ and $x\left(E[U]\right)\leq\lfloor\frac{1}{2}|U|\rfloor$ in the description of the matching polytope $MP(G_1)$ in (\ref{eq:MatchPoly}) for all $v\in V(G_1)$ and for all $U\subseteq V(G_1)$ with $|U|$ odd, and $|U|\geq 3$, repectively. From our discussion in Remark \ref{rem:Rem1}, the Lagrange mulitipliers corresponding to inequalities $x_e\geq 0$ are all zero.

Set
\[
g(\mathbf x) = -\sum_{e\in E(G_1)} p_e\log x_e,
\]
Then the Lagrangian of $g(\mathbf x)$ is
\begin{eqnarray}
L\left(\mathbf x, \mathbf{\lambda}, \mathbf{\gamma}\right) &&=-\sum_{e\in E(G_1)}p_e\log x_e +\sum_{e=\{u,v\}}\left(\lambda_u+\lambda_v\right)\left(x_e - 1\right)\nonumber\\
&&+\sum_{e\in E(G_1)}\sum_{\substack{U\subseteq V,\\U\ni e,|U|~\text{odd},~|U|\geq 3 }}\gamma_U x_e - \sum_{\substack{U\subseteq V,\\|U|~\text{odd},~|U|\geq 3}}\lfloor\frac{1}{2}|U|\rfloor,
\end{eqnarray}
Using KKT conditions (see S. Boyd, and L. Vanderberghe \cite{Boyd}), the vector $\mathbf x^*$ minimizes $g(\mathbf x)$ if and only if it satisfies \begin{eqnarray}\label{eq:KKT2}
&&\frac{\partial L}{\partial x_e^*} = 0,\nonumber\\
&&\rightarrow -\frac{p_e}{x_e^*} + \left(\lambda_u+\lambda_v\right) + \sum_{\substack{U\subseteq V,\\U\ni e, |U|~\text{odd},~|U|\geq 3}}\gamma_U = 0~\text{for}~e=\{u,v\}.
\end{eqnarray}
Fix the probability density to be uniform over the edges of $G_1$, that is
\[
p_e = \frac{1}{m}, ~~~\forall e\in E(G_1),
\]
Note that the vector $\frac{\mathbf 1}{k}$ is a feasible point in the matching polytope $MP(G_1)$. Now, one can verify that specializing the variables as
\begin{eqnarray}
&&\mathbf x^* = \frac{\mathbf 1}{k},\nonumber\\
&&\gamma_U = 0~~~~~~~~~~~\forall U\subseteq V,~|U|~\text{odd},~|U|\geq3\nonumber\\
&&\lambda_u=\lambda_v = \frac{k}{2m}~~~~~~~~~~\forall~e=\{u,v\}.\nonumber
\end{eqnarray}
satisfies the equations (\ref{eq:KKT2}). Thus
\[
H_k\left(G_2,\mathbf u\right) = \log k.
\]
Then using Lemma \ref{lem:keylemma} and the assumption $\chi_f(G_2) = k$ the theorem is proved.
\qed

It is well known that cubic graphs has a lot of interesting structures. For example, it can be checked that every edge in a bridgeless cubic graph is in a perfect matching. Furthermore, L. Lov\'{a}sz and M. D. Plummer \cite{Lov3} conjectured that every bridgeless cubic graph has an exponentially many perfect matching. This conjecture was proved by Louis Esperet, et. al. \cite{Esp} recently.
Now we have the following interesting statement for every cubic bridgeless graph.
\begin{corollary}\label{cor:symmcubic1}
The line graph of every cubic bridgeless graph $G_1=(V_1,E_1)$ is symmetric with respect to graph entropy.
\end{corollary}
\proof
We may assume that $G_1$ is connected. Let $U\subseteq V_1$ and let $U_1\subseteq U$ consist of vertices $v$ such that $\delta(v)\cap\delta(U)=\emptyset$. Then using handshaking lemma for $G_1[U]$, we have
\[
3|U_1| + 3|U\setminus U_1| - |\delta(U)| = 2|E(G_1[U])|.
\]
And consequently,
\[
3|U| = |\delta(U)| \mod 2,
\]
Assuming $|U|$ is odd and noting that $G_1$ is bridgeless, we have
\[
\delta(U)\geq 3.
\]
%Note that for every $U\subseteq V_1$ with $|U|$ odd (see D. Ullman and E. Scheinerman\cite{fgt} ),
%\[
%|\delta(U)|\geq 3.
%\]
Then, considering Corollary \ref{cor:kgraph}, the corollary is proved.
\qed
%\begin{figure}[!t]
%    \begin {center}
%        \begin{tikzpicture}
%        [scale = 1.5,
%        foo/.style={line width = 15pt}]
%%            \draw    (R) to[out=-20,in=-70] (B)
%            \draw (-1,1) -- (1,1);
%\draw (1,1) -- (0,2.5);
%\draw (0,2.5) -- (1,1);
%\draw (0,2.5) -- (-1,1);
%\draw (-1,1) -- (-1.6,.5);
%\draw (0,2.5) -- (0,3.2);
%\draw (1,1) -- (1.6,.5);
%\draw (-1.6,.5) -- (-2.6,.8);
%\draw (-1.6,.5) -- (-1.6,-.3);
%\draw (0,3.2) -- (-.8, 4);
%\draw (0,3.2) -- (.8,4);
%\draw (1.6,.5) -- (2.6,.8);
%\draw (1.6,.5) -- (1.6,-.3);
%\draw (-2.6,.8) -- (-.8,4);
%\draw (2.6,.8) -- (.8, 4);
%\draw (-1.6,-.3) -- (1.6,-.3);
%
%\draw    (-1.6,-.3) to[out=190,in=90] (.8,4);
%            \node[font=\small] at (0.2,1.2) {$v_1$};
%            \node[font=\small] at (.8,0) {$v_2$};
%            \node[font=\small] at (-.8,0) {$v_3$};
%            \node[font=\small] at (0,2.3) {$v_4$};
%            \node[font=\small] at (1.8,-.7) {$v_5$};
%            \node[font=\small] at (-1.8,-.7) {$v_6$};
%            \filldraw [blue]
%            (0,1) circle (3pt)
%            (.5,0) circle (3pt)
%            (-.5,0) circle (3pt);
%            \filldraw [blue]
%            (0,2) circle (3pt)
%            (1.5,-.7) circle (3pt)
%            (-1.5,-.7) circle (3pt);
%        \end{tikzpicture}
%    \end{center}
%\caption{A bridgeless cubic graph.}
%  \end{figure}

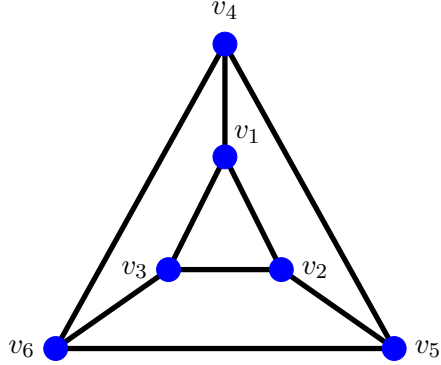
\begin{figure}[!t]
    \begin {center}
        \begin{tikzpicture}
        [scale = 1.5,
        foo/.style={line width = 15pt}]
            \draw[line width=2pt] (0,1) -- (0.5,0);
            \draw[line width=2pt] (0.5,0) -- (-.5,0);
            \draw[line width=2pt] (-.5,0) -- (0,1);
            \draw[line width=2pt] (0,1) -- (0,2);
            \draw[line width=2pt] (0,2) -- (-1.5,-.7);
            \draw[line width=2pt] (-1.5,-.7) -- (1.5,-.7);
            \draw[line width=2pt] (1.5,-.7) -- (0,2);
            \draw[line width=2pt] (1.5,-.7) -- (.5,0);
            \draw[line width=2pt] (-1.5,-.7) -- (-.5,0);
            \node[font=\small] at (0.2,1.2) {$v_1$};
            \node[font=\small] at (.8,0) {$v_2$};
            \node[font=\small] at (-.8,0) {$v_3$};
            \node[font=\small] at (0,2.3) {$v_4$};
            \node[font=\small] at (1.8,-.7) {$v_5$};
            \node[font=\small] at (-1.8,-.7) {$v_6$};
            \filldraw [blue]
            (0,1) circle (3pt)
            (.5,0) circle (3pt)
            (-.5,0) circle (3pt);
            \filldraw [blue]
            (0,2) circle (3pt)
            (1.5,-.7) circle (3pt)
            (-1.5,-.7) circle (3pt);
        \end{tikzpicture}
    \end{center}
\caption{A bridgeless cubic graph.}\label{fig:bridgelesscubic}
  \end{figure}
Figure \ref{fig:bridgelesscubic} shows a bridgeless cubic graph which is not edge transitive and its edges are not covered by disjoint copies of stars and triangles. Thus the line graph of the shown graph in Figure \ref{fig:bridgelesscubic} is neither vertex transitive nor covered by disjoint copies of its maximum size cliques. However, it is symmetric with respect to graph entropy by Corollary \ref{cor:symmcubic1}.

Figure \ref{fig:cubicgraph} shows a cubic graph with a bridge. The fractional edge chromatic number of this graph is $3.5$ while the entropy of its line graph is $1.75712$, i.e., $\log_23.5 = 1.8074>1.75712$. Thus, its line graph is not symmetric with respect to graph entropy, and we conclude that Corollary \ref{cor:symmcubic1} is not true for cubic graphs with bridge.
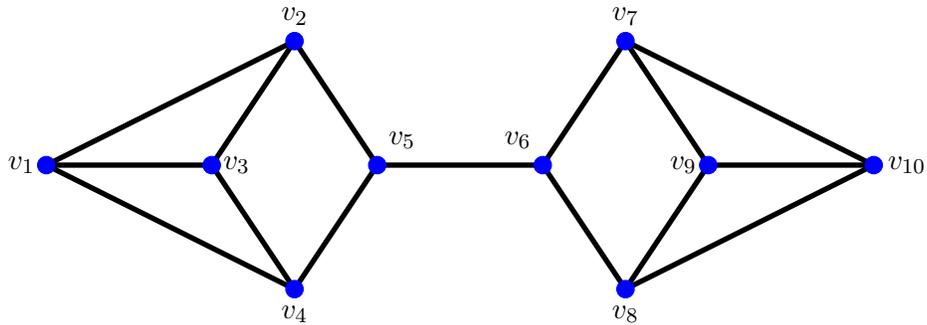
\begin{figure}[!t]\label{fig:cubicgraph}
    \begin {center}
        \begin{tikzpicture}
        [scale = 1.1,
        foo/.style={line width = 15pt}]
            \draw[line width=2pt] (-5,0) -- (-3,0);
            \draw[line width=2pt] (-5,0) -- (-2,1.5);
            \draw[line width=2pt] (-5,0) -- (-2,-1.5);
            \draw[line width=2pt] (-3,0) -- (-2,1.5);
            \draw[line width=2pt] (-3,0) -- (-2,-1.5);
            \draw[line width=2pt] (-2,1.5) -- (-1,0);
            \draw[line width=2pt] (-2,-1.5) -- (-1,0);
            \draw[line width=2pt] (-1,0) -- (1,0);
            \draw[line width=2pt] (1,0) -- (2,1.5);
            \draw[line width=2pt] (1,0) -- (2,-1.5);
            \draw[line width=2pt] (2,1.5) -- (3,0);
            \draw[line width=2pt] (2,-1.5) -- (3,0);
            \draw[line width=2pt] (3,0) -- (5,0);
            \draw[line width=2pt] (2,1.5) -- (5,0);
            \draw[line width=2pt] (2,-1.5) -- (5,0);

            \node[font=\small] at (-5.3,0) {$v_1$};
            \node[font=\small] at (-2,1.8) {$v_2$};
	    \node[font=\small] at (-2.7,0) {$v_3$};
           \node[font=\small] at (-2,-1.8) {$v_4$};
           \node[font=\small] at (-.7,.3) {$v_5$};
           \node[font=\small] at (.7,.3) {$v_6$};
           \node[font=\small] at (2,1.8) {$v_7$};
           \node[font=\small] at (2,-1.8) {$v_8$};
           \node[font=\small] at (2.7,0) {$v_9$};
           \node[font=\small] at (5.4,0) {$v_{10}$};

            \filldraw [blue]
            (-5,0) circle (3pt)
	    (-2,1.5) circle (3pt)
            (-3,0) circle (3pt)
            (-2,-1.5) circle (3pt)
            (-1,0) circle (3pt)
            (1,0) circle (3pt)
            (2,1.5) circle (3pt)
            (2,-1.5) circle (3pt)
            (3,0) circle (3pt)
            (5,0) circle (3pt);

        \end{tikzpicture}
    \end{center}
\caption{A cubic one-edge connected graph.}\label{fig:cubicgraph}
  \end{figure}

\chapter{Future Work}

In this chapter we explain two possible research directions related to the entropy of graphs discussed in previous chapters. Since these directions are related to a superclass of perfect graphs which are called \emph{normal graphs}\index{normal graph} and \emph{Lov\'{a}sz} $\vartheta$\index{Lov\'{a}sz $\vartheta$}, we explain the corresponding terminologies and results in the sequel.

\section{Normal Graphs}
Let $G$ be a graph. A set $\mathcal A$ of subsets of $V(G)$ is a \emph{covering}, if every vertex of $G$ is contained in an element of $\mathcal A$.

We say that graph $G$ is \emph{Normal} if there exists two coverings $\mathcal C$ and $\mathcal S$ such that every element $C$ of $\mathcal C$ is a clique and every element $S$ of $\mathcal S$ is an independent set and the intersection of any element of $\mathcal C$ and any element of $\mathcal S$ is nonempty, i.e.,
\[
C\cap S \neq \emptyset,~\forall C\in \mathcal C,~S\in\mathcal S.
\]
Recall from the sub-additivity of Graph Entropy, we have
\begin{equation}\label{eq:subadd6}
H(P)\leq H_k(G,P) + H_k(\overline G,P).
\end{equation}
A probabilistic graph $(G,P)$ is \emph{weakly splitting} if there exists a nowhere zero probability distribution $P$ on its vertex set which makes inequality (\ref{eq:subadd6}) equality. The following lemma was proved in J. K\"{o}rner et. al. \cite{JKor2}.
\begin{lemma}\label{lem:normal1}\emph{(J. K\"{o}rner, G. Simonyi, and Zs. Tuza)}
A graph $G$ is weakly splitting if and only if it is normal.
\end{lemma}
Furthermore, we call $(G,P)$ is strongly splitting if inequality (\ref{eq:subadd6}) becomes equality for every probability distribution $P$. The following lemma was proved in I. Csisz\'{a}r et. al. \cite{Csis}.
\begin{lemma}\label{lem:normal2}\emph{(I. Csisz\'{a}r et. al.)}
For a probabilistic graph $(G,P)$, we have
\[
H(P) = H_k(G,P) + H_k(\overline G, P)~\text{if and only if}~H_{VP(G)}(P) = H_{FVP(G)}(P).
\]
\end{lemma}
Furthermore, it is shown in I. Csisz\'{a}r et. al. \cite{Csis} that
\begin{lemma}\label{lem:normal3}\emph{(I. Csisz\'{a}r et. al.)} A graph $G$ is perfect if and only
\[
H_{VP(G)}(P) = H_{FVP(G)}(P).
\]
\end{lemma}
Using Lemmas \ref{lem:normal1}, \ref{lem:normal2}, and \ref{lem:normal3}, we conclude that every perfect graph is also a normal graph. This fact was previously proved in J. K\"{o}rner \cite{JKor01}. It is shown in \cite{ZPat} that the line graph of a cubic graph is normal. Furthermore, it is shown in J. K\"{o}rner \cite{JKor01} that every odd cycle of length at least nine is normal. Smaller odd cycles are either perfect like a triangle or not perfect nor normal like $C_5$ and $C_7$. If we require that every induced subgraph of a normal graph to be normal, we obtain the notion of hereditary normality. The following conjecture was proposed in C. De Simone and J. K\"{o}rner \cite{De}.
\begin{conj}\emph{Normal Graph Conjecture}
A graph is hereditarily normal if and only if the graph nor its complement contains $C_5$ or $C_7$ as an induced subgraph.
\end{conj}
A \emph{circulant} $C_n^k$ is a graph with vertex set $\{1,\cdots,n\}$, and two vertices $i\neq j$ are adjacent if and only if
\[
i - j \equiv k~\mathrm{mod}~n.
\]
We assume $k\geq 1$ and $n\geq 2(k+1)$ to avoid cases where $C_n^k$ is an independent set or a clique. The following theorem was proved in L. E. Trotter, jr. \cite{Trot}.
\begin{theorem}\emph{(L. E. Trotter, jr.)}
The circulant $C_{n^\prime}^{k^\prime}$ is an induced subgraph of $C_n^k$ if and only if \[
\frac{k+1}{k^\prime+1}n^\prime\leq n\leq\frac{k}{k^\prime}.
\]
\end{theorem}
Note that
\[
C_{n^\prime}^{k^\prime}\subset C_n^k,
\]
implies $k^\prime<k$ and $n^\prime<n$. Particularly, the following lemma was proved in A. K. Wagler \cite{Wag}.
\begin{lemma}\emph{(A. K. Wagler)}
\begin{enumerate}[(i)]
  \item $C_5\subseteq C_n^k$ if and only if $\frac{5(k+1)}{2}\leq n\leq 5k$.
  \item $C_7\subseteq C_n^k$ if and only if $\frac{7(k+1)}{2}\leq n\leq 7k$.
  \item $C_7^2\subseteq C_n^k$ if and only if $\frac{7(k+1)}{3}\leq n\leq \frac{7k}{2}$.
\end{enumerate}
Using the above theorem and lemma, A. K. Wagler \cite{Wag} proved the Normal Graph Conjecture for circulants $C_n^k$.
\end{lemma}
One direction for future research is investigating the Normal Graph Conjecture for general circulants and Cayley graphs.
\section{Lov\'{a}sz $\vartheta$ Function and Graph Entropy}

An old problem in information and graph theory is to determine the \emph{zero error Shannon capacity} $C(G)$ of a graph $G$.
Let $G$ be a graph with vertex set $V(G)$ and edge set $E(G)$. The $n$-th normal power of $G$ is the graph $G^n$ with vertex set $V(G^n)=\left(V(G)\right)^n$ and two vertices $(x_1,\cdots,x_n)\neq(y_1,\cdots,y_n)$ are adjacent if and only if
\[
x_i = y_i~\text{or}~\{x_i,y_i\}\in E(G)~\forall i\in\{1,\cdots,n\}.
\]
The \emph{zero error Shannon capacity} $C(G)$ of a graph $G$ is defined as
\[
C(G) = \limsup_{n\rightarrow\infty} \frac{1}{n}\log\alpha\left(G^n\right).
\]
Let $P$ denote the probability distribution over the vertices of $G$, and $\epsilon>0$. Let $x^n\in V^n$ be an $n$-sequence whose entries are from $\mathcal X$, and $N\left(a|x^n\right)$ denote the number of occurrences of an element $a\in\mathcal X$. We call the set of $\left(P,\epsilon\right)$-typical sequences $T^(P,\epsilon)$ to be the set of $n$-sequences $x^n\in V^n$ such that
\[
|N(a|x^n)-P(X = a)|\leq n\epsilon.
\]
Then the \emph{capacity of the graph relative to} $P$ is
\[
C(G,P)=\lim_{\epsilon\rightarrow0}\limsup_{n\rightarrow\infty}\frac{1}{n}\log
\alpha\left(G^(P,\epsilon)\right).
\]
Given a probabilistic graph $(G,P)$, K. Marton in K. Marton \cite{Mart} introduced a functional $\lambda(G,P)$ which is analogous to Lov\'{a}sz's bound $\vartheta(G)$ on Shannon capacity of graphs. Similar to $\vartheta(G)$, the probabilistic functional $\lambda(G,P)$ is based on the concept of \emph{orthonormal representation} of a graph which is recalled here.

Let $U=\{\mathbf u_i:i\in V(G)\}$ be a set of unit vectors of a common dimension $d$ such that
\[
\mathbf u_i^T\mathbf u_j = 0~\text{if}~i\neq j~\text{and}~\{i,j\}\notin E(G).
\]
Let $\mathbf c$ be a unit vector of dimension $d$. Then, the system $\left(U,\mathbf c\right)$ is called an orthonormal representation of the graph $G$ with handle $\mathbf c$.

Letting $T(G)$ denote the set of all orthonormal representations with a handle for graph $G$, L. Lov\'{a}sz \cite{Lov} defined
\[
\vartheta(G)=\min_{\left(U,\mathbf c\right)\in T(G)}\max_{i\in V(G)}\frac{1}{(\mathbf u_i,\mathbf c)^2}.
\]
Then it is shown in L. Lo\'{a}sz \cite{Lov} that zero error Shannon capacity $C(G)$ can be bounded above by $\vartheta(G)$ as
\[
C(G)\leq\log \vartheta(G).
\]
A probabilistic version of $\vartheta(G)$ denoted by $\lambda(G,P)$ is defined in K. Marton \cite{Mart} as
\[
\lambda(G,P):=\min_{\left(U,\mathbf c\right)\in T(G)}\sum_{i\in V(G)}P_i\log\frac{1}{(\mathbf u_i,\mathbf c)^2}.
\]
K. Marton \cite{Mart} showed that
\begin{theorem}\emph{(K. Marton)}
The capacity of a probabilistic graph $(G,P)$ is bounded above by $\lambda(G,P)$, i.e.,
\[
C(G,P)\leq \lambda(G,P).
\]
\end{theorem}
The following theorem was proved in K. Marton \cite{Mart} which relates $\lambda(G,P)$ to $H_k(G,P)$.
\begin{theorem}\emph{(K. Marton)}
For any probabilistic graph $(G,P)$,
\[
\lambda\left(\overline G, P\right)\leq H_k\left(G,P\right).
\]
Furthermore, equality holds if and only if $G$ is perfect.
\end{theorem}
K. Marton \cite{Mart} also related $\lambda(G,P)$ to $\vartheta(G)$ by showing
\begin{equation}\label{eq:Hom1}
\max_{P}\lambda(G,P) = \log\vartheta(G).
\end{equation}
It is worth mentioning that $\vartheta(G)$ can be defined in terms of graph homomorphisms as follows.

Let $d\in\mathbb{N}$ and $\alpha<0$. Then we define $S(d,\alpha)$ to be an infinite graph whose vertices are unit vectors in $\mathbb{R}^d$. Two vertices $\mathbf u$ and $\mathbf v$ are adjacent if and only if $\mathbf u\mathbf v^T = \alpha$. Then
\begin{equation}\label{eq:Hom2}
\vartheta\left(\overline G\right)=\min\left\{1 - \frac{1}{\alpha}:G\rightarrow S\left(d,\alpha\right),~\alpha<0\right\}.
\end{equation}
Thus, noting (\ref{eq:Hom1}) and (\ref{eq:Hom2}) and the above discussion, investigating the relationship between graph homomorphism and graph entropy which may lead to investigating the relationship between graph homomorphism and graph covering problem seems interesting.

%\appendix
%% Add a title page before the appendices and a line in the Table of Contents
%\chapter*{APPENDICES}
%\addcontentsline{toc}{chapter}{APPENDICES}

\newtheorem*{mydef1}{A. 1. Lemma}
\newtheorem*{mydef2}{A. 2. Lemma}
\newtheorem*{mydef3}{A. 3. Lemma}
\chapter*{Appendix A}
\addcontentsline{toc}{chapter}{Appendix A}
\section*{Proof of Lemma 3.2.2}
First, we state a few lemmas as follow.
\begin{mydef1}\label{lem:kernel}
The chromatic number of a graph $G$, i.e., $\chi(G)$ is equal to the minimum number of maximal independent sets covering $G$.
\end{mydef1}
\proof

Let $\kappa(G)$ be the minimum number of maximal independent sets covering the vertices of $G$. Then $\kappa(G)\leq\chi(G)$, since the colour classes of any proper colouring of $V(G)$ can be extended to maximal independent sets. On the other hand, consider a covering system consisting of maximal independent sets $\mathcal S$ with a minimum number of maximal independent sets. Let $\mathcal S = \{S_1,\cdots,S_{\kappa(G)}\}$. We define a colouring $c$ of the vertices of graph $G$ as
\[
c(v) = i,~\forall{v\in S_i\setminus S_{i-1}},~\text{and}~\forall i\in\{1,\cdots,\kappa(G)\},
\]
The proposed colouring is a proper colouring of the vertices of $V(G)$ in which each colour class corresponds to a maximal independent set in our covering system $\mathcal S$. That is
\[
\kappa(G)\geq\chi(G).
\]
\qed

Let $\mathcal X$ be a finite set and let $P$ be a probability density on its elements. Let $K$ be a constant. Then, a sequence $\mathbf x\in\mathcal X^n$ is called $P$-\emph{typical} if for every $y\in\mathcal X$ and for the number of occurrences of the element $y$ in $\mathbf x$, i.e., $N(y|\mathbf x)$, we have
\[
|N(y|\mathbf x) - np(y)|\leq K\sqrt{p(y)}.
\]
Then we have the following lemma.
\begin{mydef2}\label{lem:type}
Let $T^n(P)$ be the set of the $P$-typical $n$-sequences. Then,
\begin{enumerate}[(i)]
  \item For all $\epsilon>0$ there exists $K>0$ such that
	\[
	P\left(\overline{T^{n}(P)}\right)<\epsilon,\quad\text{for this $K$}.
	\]
  \item For every typical sequence $\mathbf x$ we have
	\[
	2^{-(nH(P)+C\sqrt{n})}\leq P\left(\mathbf x\right)\leq 2^{-(nH(P)-C\sqrt{n})},
	\]
       for some constant $C>0$ depending on $|\mathcal X|$ and $\epsilon>0$ and independent of $n$ and $P$,
  \item The number of typical sequences $N(n)$ is bounded as
	\[
        2^{nH(P) - C\sqrt{n}}\leq N(n)\leq 2^{nH(P) + C\sqrt{n}}.
	\]
        for some constant $C>0$ depending on $|\mathcal X|$ and $\epsilon>0$ and independent of $n$ and $P$.
\end{enumerate}
\qed
\end{mydef2}
Having $\mathcal X$ defined as above, let $\left(G,P\right)$ be a probabilistic graph with vertex $V(G) = \mathcal X$. We define the relation $e$ as
\[
x e y\Longleftrightarrow \text{either} \{x,y\}\in E(G)~\text{or}~x = y.
\]
If $e$ determines an equivalence relation on the vertex set $V(G)$, then graph $G$ is the union of pairwise disjoint cliques. Let $H(P|e)$ denote the conditional entropy given the equivalence class $e$, i.e.,
\[
H(P|e) = \sum_{x\in \mathcal X}p(x)\log\frac{\sum_{y:x e y}p(y)}{p(x)},
\]
Let $\mathcal A$ denote the collection of equivalence classes under $e$. Let $P_e$ be the probability density on the elements $A$ of $\mathcal A$ given by
\[
p_e(A) = \sum_{x\in A}p(x),
\]
Then we have the following lemma (see V. Anantharam \cite{Anan} and J. K\"{o}rner \cite{JKor} ).
\begin{mydef3}(V. Anantharam).\label{mydef3:EquiBound}
The number of $P$-typical $n$-sequences in a $P_e$-typical $n$-sequence of equivalence classes is bounded below by $2^{nH(P|P_e) - C\sqrt{n}}$ and bounded above by $2^{nH(P|P_e) + C\sqrt{n}}$ for some constant $C>0$.
\end{mydef3}
\proof

Let $\mathbf A  = (A_1,\cdots,A_n)$ be a $P_e$-typical $n$-sequence. That is for each $A\in\mathcal A$
\begin{equation}\label{eq:numbound}
|N(A|\mathbf A) - np_e(A)|\leq K\sqrt{np_e(A)},
\end{equation}
Then for all $A\in\mathcal A$, we have
\begin{equation}\label{eq:typdef}
np(A)\leq\max(4K^2,2N(A|\mathbf A)),
\end{equation}
The proof is as follows. Suppose $np(A)\geq 4K^2$. Then $np(A)\geq 2K\sqrt{np_e(A)}$ and therefore,
\[
N(A|\mathbf A)\geq np_e(A) - K\sqrt{np_e(A)}\geq \frac{np_e(A)}{2},
\]
Let $\mathbf x = (x_1,\cdots,x_n)$ be a $P$-typical $n$-sequence in $\mathbf A$, i.e.,
\[
x_i\in A_i,~1\leq i\leq n.
\]
From $P$-typicality of $\mathbf x$, we have
\[
|N(x|\mathbf x) - np(x)|\leq K\sqrt{np(x)}.
\]
Now, we prove that for each $A\in\mathcal A$, the restriction of $\mathbf x$ to those co-ordinates having $A_i = A$ is $\left(\frac{p(x)}{p_e(A)}:x\in A\right)$-typical. For $x\in A$, we have
\begin{eqnarray}
\left|N(x|\mathbf x) - N(A|\mathbf A)\frac{p(x)}{p_e(A)}\right|&\leq&\left|N(x|\mathbf x) - np(x)\right|+\left|np(x)-N(A|\mathbf A)\frac{p(x)}{p_e(A)}\right|\nonumber\\
&\leq&K\sqrt{np(x)}+\frac{p(x)}{p_e(A)}K\sqrt{np_e(A)}\nonumber\\
&=&K\left(\sqrt{\frac{p(x)}{p_e(A)}}+\frac{p(x)}{p_e(A)}\right)\sqrt{np_e(A)}.\nonumber
\end{eqnarray}
Using (\ref{eq:typdef}), and noting $N(A|\mathbf A)\geq 1$ and $\frac{p(x)}{p_e(A)}\leq\sqrt{\frac{p(x)}{p_e(A)}}$, we get
\begin{eqnarray}
\left|N(x|\mathbf x)-N(A|\mathbf A)\frac{p(x)}{p_e(A)}\right|
&\leq&K\left(\sqrt{\frac{p(x)}{p_e(A)}}+\frac{p(x)}{p_e(A)}\right)\sqrt{\max(4K^2,2N(A|\mathbf A))}\nonumber\\
&\leq&K\left(\sqrt{\frac{p(x)}{p_e(A)}}+\frac{p(x)}{p_e(A)}\right)\sqrt{\max\left(\frac{4K^2}{N(A|\mathbf A)},2\right)}.\sqrt{N(A|\mathbf A)}\nonumber\\
&\leq&\max(2K^2,\sqrt{2}K)\left(\sqrt{\frac{p(x)}{p_e(A)}}
+\frac{p(x)}{p_e(A)}\right)\sqrt{N(A|\mathbf A)}\nonumber\\
&\leq&2\max(2K^2,\sqrt{2}K)\sqrt{\frac{N(A|\mathbf A)p(x)}{p_e(A)}}.
\end{eqnarray}
Now, letting $H(P|e=A)$ denote $\sum_{x\in A}\frac{p(x)}{p_e(A)}\log\frac{p_e(A)}{p(x)}$, we give the following lower and upper bounds on the number of $P$-typical $n$-sequences $\mathbf x$ in $\mathbf A$. Let $C>0$ be some constant depending on $K$ and $|\mathcal X|$ as in Lemma \ref{lem:type}, then using Lemma \ref{lem:type} and (\ref{eq:numbound}) we get the following upper bound on the $P$-typical $n$-sequences $\mathbf x$ in $\mathbf A$
\begin{eqnarray}
&&\prod_{A\in\mathcal A}2^{N(A|\mathbf A)H(\frac{P}{P_e(A)})}+ C\sqrt{N(A|\mathbf A)}\nonumber\\
&&=2^{n\sum_{A\in\mathcal A}\left(\frac{N(A|\mathbf A)}{n}H(P|e=A)+C\sqrt{N(A|\mathbf A)}\right)}\nonumber\\
&&\leq 2^{n\sum_{A\in\mathcal A}\left(p_e(A)+\frac{K}{n}\sqrt{np_e(A)}\right)H(P|e=A)+\sum_{A\in\mathcal A}C\sqrt{n}}\nonumber\\
&&\leq 2^{n\sum_{A\in\mathcal A}p_e(A)H\left(P|P_e\right)+K\sum_{A\in\mathcal A}\sqrt{np_e(A)}H\left(P|e = A\right)+C|\mathcal X|\sqrt{n}}\nonumber\\
&&\leq 2^{nH\left(P|P_e\right)+\sqrt{n}\left(C|\mathcal X|+K\sum_{A\in\mathcal A}\log|A|\right)}\nonumber\\
&&= 2^{nH\left(P|P_e\right)+\sqrt{n}\left(C|\mathcal X|+K|\mathcal X|\right)},\nonumber
\end{eqnarray}
Now, setting
\[
C_1 = C|\mathcal X|+K|\mathcal X|,
\]
Thus, the number of $P$-typical $n$-sequences $\mathbf x$ in $\mathbf A$ is upper bounded by
\[
2^{nH(P|e) + C_1\sqrt{n}},
\]
Similarly, the number of $P$-typical $n$-sequences $\mathbf x$ in $\mathbf A$ is lower bounded by
\[
2^{nH(P|e) - C_1\sqrt{n}}.
\]
\qed

\prooflem

Let $0<\epsilon<1$, and $M(n,\epsilon)$ denote
\[
\min_{U\in T_\epsilon^{(n)}}  \chi(G^{(n)}[U]),
\]
for sufficiently large $n$. Let $\lambda>0$ be a positive number. First, we show that
\[
M(n,\epsilon)\geq 2^{\left(H^\prime(G,P)-\lambda\right)}.
\]
Consider $G^{(n)}[U]$ for some $U\in T_\epsilon^{(n)}$. Using Lemma A. 2, for any $\delta>0$ there is a $K>0$ such that for any sufficiently large $n$, we have
\[
P\left(T^n(P)\right)\geq 1 - \delta.
\]
First, note that
\begin{equation}\label{eq:Bound2}
1 - \delta - \epsilon \leq P\left(U\cap T^n(P)\right).
\end{equation}
Now, we estimate the chromatic number of $G^{(n)}[U\cap T^n(P)]$. Let $\mathcal S^n$ denote the family of the maximal independent sets of $G^{(n)}$. Note that every colour class in a minimum colouring of graph can be enlarged to a maximal independent set. Thus,
\begin{equation}\label{eq:Bound3}
P\left(U\cap T^n(P)\right)\leq \chi\left(G^{(n)}[U\cap T^n(P)]\right).\max_{\mathbf S\in\mathcal S^n} P\left(\mathbf S\cap T^n(P)\right),
\end{equation}
Furthermore, we have
\begin{equation}\label{eq:Bound4}
\max_{\mathbf S\in\mathcal S^n}P\left(\mathbf S\cap T^n(P)\right)\leq\max_{\mathbf x\in T^n(P)}p(\mathbf x).\max_{S\in\mathcal S^n}|\mathbf S\cap T^n(P)|.
\end{equation}
It is worth mentioning that $|\mathbf S\cap T^n(P)|$ is the number of typical sequences contained in $\mathbf S$. Furthermore, note that  $\mathbf S$ can be considered as an $n$-sequence of maximal independent sets taken from $\mathcal S$.

Let $N(y,R|\mathbf x,\mathbf S)$ denote the number of occurrences of the pair $(y,R)$ in the following double $n$-sequence
\begin{equation*}
\left(
\begin{array}{cccc}
x_1 & x_2 & \cdots & x_n \\
S_1 & S_2 & \cdots & S_n
\end{array} \right)
\end{equation*}
In other words, $N(y,R|\mathbf x,\mathbf S)$ is the number of occurrences of the letter $y$ selected from the maximal independent set $R$ in the $n$-sequence $\mathbf x$ taken from the maximal independent sequence $\mathbf S$. Similarly, $N(y|\mathbf x)$ denotes the number of occurrences of the source letter $y$ in the $n$-sequence $\mathbf x$.

Setting
\begin{equation}\label{eq:aux1}
q(y,R) = \frac{N(y,R|\mathbf x,\mathbf S)}{N(y|\mathbf x)}.p(y),
\end{equation}
we have
\begin{eqnarray}\label{eq:bound1}
&&|N(y,R|\mathbf x,\mathbf S) - n q(y,R)| = \left|\frac{n q(y,R)}{np(y)}\right|.|N(y|\mathbf x)-np(y)|\nonumber\\
&&\leq\left|\frac{q(y,R)}{p(y)}\right|.K\sqrt{np(y)} = K\sqrt{n.\frac{q^2(y,R)}{p(y)}}\leq K\sqrt{n q(y,R)},\nonumber
\end{eqnarray}
since $\mathbf x$ is a $P$-typical sequence. Let
\begin{equation}\label{eq:aux2}
a(R)=\sum_{y:y\in R}q(y,R).
\end{equation}
Then
\[
N(R|\mathbf S) - na(R) = \sum_{y\in R}N(y,R|\mathbf x, \mathbf S) - nq(y,R),
\]
And therefore using \ref{eq:bound1},
\begin{eqnarray}
|N(R|\mathbf S) - n a(R)|&&\leq \sum_{y\in\mathcal X}K\sqrt{n q(y,R)}\leq K_1\sqrt{n\sum_{y\in\mathcal X}q(y,R)}\nonumber\\
&&= K_1\sqrt{n a(R)}.
\end{eqnarray}
Now, we define an auxiliary graph $\Gamma$ of $G$ as follows. Letting $S$ be a maximal independent set of $G$ containing a vertex $x$ of $G$, the vertex set of $\Gamma$ consists of pairs $(x,S)$. Furthermore, two vertices $(x,S)$ and $(y,R)$ are adjacent if and only if $S\neq R$.
Let $K_2>0$ be some constant. Then, applying Lemma A.3 with the equivalence relation $a$ which is
\[
\left((x,S),(y,R)\right)\notin E\left(\Gamma\right),
\]
and probability density $Q$ for the graph $\overline{\Gamma}$, the number of $Q$-typical $n$-sequences in each $a$-typical equivalence class $A$ which is a maximal independent set of $G$ lies in the interval
\begin{equation}\label{eq:interval}
\left[2^{nH\left(Q|a\right)-K_2\sqrt{n}},2^{nH\left(Q|a\right)+K_2\sqrt{n}}\right].
\end{equation}
Noting that every pair $(y,R)$ may occur zero, one,$\cdots$, or $n$-times in the $n$-sequence $(\mathbf x, \mathbf S)$ and for a given $y$ knowing $N(y,R|\mathbf x,\mathbf S)$ for all $R$ uniquely determines $N(y|\mathbf x)$, there are at most $(n+1)^{|V(\Gamma)|}$ different auxiliary densities of the type given by (\ref{eq:aux1}). Now we bound $\max_{\mathbf S\in\mathcal S^n}|\mathbf S\cap T^{n}(P)|$ as follows.
Note that $\mathbf S\cap T^{n}(P)$ is the set of $P$-typical $n$-sequenences which are contained in a given maximal independent set $\mathbf S$ in $G^{(n)}$. Then letting $\mathcal Q$ be the feasible joint distribution for $(X,S)$, for all $\mathbf S\in \mathcal S^n$ and all $Q\in\mathcal Q$, set
\[
T^{n}(S,Q):=\{\mathbf x:\mathbf x\in\mathcal X^n, x_i\in S_i, (\mathbf x,\mathbf S)~\text{is}~Q\text{-typical.}\}
\]
%Every $n$-sequence of $X^n$ in $\mathbf S\cap T^n(\mathbf p)$ is $\mathbf Q$-typical for a certain distribution of this type, and therefore we have
From (\ref{eq:aux1}), for all $\mathbf S\in \mathcal S^n$ and for all $\mathbf x$ in $|\mathbf S\cap T^n(P)|$ there is some $Q\in\mathcal Q$ such that $\mathbf x\in T^{n}(S,Q)$. Therefore, for all $\mathbf S\in\mathcal S^n$, we get
\begin{eqnarray}
|\mathbf S\cap T^n(P)|&\leq&|\bigcup_{Q\in\mathcal Q}T^n(S,Q)|\nonumber\\
&\leq&\sum_{Q\in\mathcal Q}|T^n(S,Q)|\nonumber\\
&\leq&|\mathcal Q|\max_{Q\in\mathcal Q}|T^n(S,Q)|,\nonumber
\end{eqnarray}
Then, using (\ref{eq:interval}),  we obtain
\begin{equation}\label{eq:Bound10}
\max_{\mathbf S\in\mathcal S^n}|\mathbf S\cap T^n(P)|\leq(n+1)^{|V(\Gamma)|}.2^{n.\max_{Q^\prime\in\mathcal Q}H(Q^\prime|a)+K_2\sqrt{n}}.
\end{equation}
Further,
\[
\sum_{R:y\in R}q(y,R)=\frac{p(y)}{N(y|\mathbf x)}.\sum_{R:y\in R}N(y,R|\mathbf x,\mathbf S) = p(y).
\]
From the Lemma A.2 part (ii), we get
\begin{equation}\label{eq:Bound11}
\max_{\mathbf x\in T^n(P)}p(\mathbf x)\leq 2^{-(nH(P) - C\sqrt{n})}.
\end{equation}
Thus, using the inequalities (\ref{eq:Bound2})-(\ref{eq:Bound4}), (\ref{eq:Bound10}) and (\ref{eq:Bound11}) we have
\begin{eqnarray}
(1 - \lambda - \epsilon)&&\leq \chi\left(G^{(n)}[U\cap T^n(\mathbf p)]\right)\nonumber\\
&&.\mathrm{exp_2}\left(n.\left(\max_{Q^\prime\in\mathcal Q}H(Q^\prime|a) - H(P)\right) + K_2\sqrt{n}+|V(\Gamma)|.\log_2(n+1)\right),\nonumber
\end{eqnarray}
And consequently,
\begin{eqnarray}\label{eq:Bound12}
\chi\left(G^{(n)}[U\cap T^n(P)]\right)&&\geq(1 - \lambda - \epsilon)\\
&&.\mathrm{exp_2}\left(n(H(P) - \max_{Q^\prime\in\mathcal Q}H(Q^\prime|a) - K_2\sqrt{n}-|V(\Gamma)|.\log_2(n+1))\right).\nonumber
\end{eqnarray}
Note that
\[
H(P) - \max_{Q^\prime\in\mathcal Q}H(Q^\prime|a) = \min_{Q^\prime\in\mathcal Q}\sum_{x,S}q^\prime(x,S)\log_2\frac{q^\prime(x,S)}{p(x).q^\prime(S)} = \min_{ Q^\prime\in\mathcal Q}I(Q^\prime).
\]
Now, considering
\[
\chi\left(G^{(n)}[U]\right)\geq\chi\left(G^{(n)}[U\cap T^n(P)]\right),
\]
and using (\ref{eq:Bound12}), for every $U\in T_{\epsilon}^{(n)}$ we get
\[
\chi(G^{(n)}[U])\geq (1 - \lambda - \epsilon).\mathrm{exp_2}\left(nH^\prime(G,P) - K_2\sqrt{n}-|V(\Gamma)|.\log_2(n+1)\right).
\]
Thus,
\[
\frac{1}{n}\log_2\left(\min_{U\in T_{\epsilon}^{n}}\chi\left(G^{(n)}[U]\right)\right)\geq\frac{1}{n}\log_2\left(1 -\lambda -\epsilon\right)+ H^\prime\left(G,P\right) - \frac{K_2}{\sqrt{n}}-\frac{|V(\Gamma)|}{n}\log_2(n+1),
\]
Therefore, we get
\begin{equation}\label{eq:Bound0}
\liminf_{n\rightarrow\infty}\frac{1}{n}\log_2M\left(n,\epsilon\right)\geq H^\prime(G,P).
\end{equation}
Now we show that for every $0<\epsilon<1$ and $\delta>0$ and sufficiently large $n$, there exists subgraphs $G^{(n)}[U]$ of $G^{(n)}$, for some $U\subseteq V(G^{(n)})$, such that
\[
\chi\left(G^{(n)}[U]\right)\leq 2^{n(H^\prime(G,P)+\delta)}.
\]
Let $Q^{*}$ be the joint density on vertices and independent sets of $G$ which minimizes the mutual information $I(Q^{*})$. That is
\[
I(Q^{*}) = H^\prime(G,P).
\]
Then the probability of every maximal independent set $S$ is
\[
Q^*(S)=\sum_{y:y\in S}Q^*(y,S).
\]
Letting $\mathbf S$ be $\mathbf S = \left(S_1,S_2, \cdots, S_n\right)\in\mathcal S^n$, we have
\[
\mathbf Q^*(\mathbf S) = \prod_{i=1}^nQ^*(S_i),
\]
Let $L$ be a fixed parameter. For a family of $L$ maximal independent sets, not necessarily distinct and not necessarily covering, we define the corresponding probability density $\mathbf Q_L^*$ as follows. We assume that the $L$ maximal independent sets of a given system of maximal independent sets are chosen independently. Thus,
\[
Q_L^*\left(\mathbf S_1,\mathbf S_2,\cdots,\mathbf S_L\right) = \prod_{j=1}^LQ^*(\mathbf S_j).
\]
Now consider a fixed $n$. Let $G^{(n)}$ be the $n$-th conormal power graph of graph $G$.  Consider systems of maximal independent sets consisting of $L$ maximal independent sets each in the form of a $n$-sequence of maximal independent sets. We call this system of maximal independent sets an $L$-system.

For each $L$-system $\left(\mathbf S_1,\mathbf S_2,\cdots,\mathbf S_L\right)$ let $U\left(\mathbf S_1,\mathbf S_2,\cdots,\mathbf S_L\right)$ be the union of all vertices of $V(G^{(n)})$ which are not covered by the $L$-system $\left(\mathbf S_1,\mathbf S_2,\cdots,\mathbf S_L\right)$. For a given $L$, we show that the expected value of $P\left(U(\mathbf S_1,\mathbf S_2,\cdots,\mathbf S_L)\right)$ is less than $\epsilon$. This implies that there exists at least one system $\mathbf S_1,\cdots,\mathbf S_L$ covering a subgraph of $G^{(n)}$ with probability greater than or equal to $1 - \epsilon$.
For an $L$-system chosen with probability $Q^*_L$, let $Q_{L,\mathbf x}^*$ be the probability that a given $n$-sequence $\mathbf x$ is not covered by an $L$-system, that is
\begin{eqnarray}
Q_{L,\mathbf x}^* &=& Q_L^*\left(\left\{\left(\mathbf S_1,\cdots,\mathbf S_L\right):\mathbf x\in U\left(\mathbf S_1,\cdots,\mathbf S_L\right)\right\}\right)\nonumber\\
&=&\sum_{\left(\mathbf S_1,\cdots,\mathbf S_L\right)\ni\mathbf x}Q_L^*\left(\mathbf S_1,\cdots,\mathbf S_L\right).\nonumber
\end{eqnarray}
Then we have
\begin{eqnarray}\label{eq:Sum0}
E\left(P\left(U\left(\mathbf S_1,\cdots,\mathbf S_L\right)\right)\right) &=& \sum_{\mathbf S_1,\cdots,\mathbf S_L}Q_L^*\left(\mathbf S_1,\cdots,\mathbf S_L\right).P\left(U(\mathbf S_1,\cdots,\mathbf S_L)\right) \nonumber\\
&=& \sum_{\left(\mathbf S_1,\cdots,\mathbf S_L\right)}Q_L^*\left(\mathbf S_1,\cdots,\mathbf S_L\right)\left(\sum_{\mathbf x\in U(\mathbf S_1,\cdots,\mathbf S_L)}P(\mathbf x)\right)\nonumber\\
&=&\sum_{\mathbf x\in\mathcal X^n}P(\mathbf x)\left(\sum_{U(\mathbf S_1,\cdots,\mathbf S_L)\ni\mathbf x}Q_L^*\left(\mathbf S_1,\cdots,\mathbf S_L\right)\right)\nonumber\\
&=&\sum_{\mathbf x\in\mathcal X^n}P(\mathbf x).Q_{L,\mathbf x}^*.
\end{eqnarray}
For a given $\epsilon$ with $0<\epsilon<1$, by Lemma A. 2 there exists a set of typical sequences with total probability greater than or equal to $1 - \frac{\epsilon}{2}$. Then we can write the right hand of the above equation as
\begin{eqnarray}
&&\sum_{\mathbf x\in\mathcal X^n}P(\mathbf x).Q_{L,\mathbf x}^*\nonumber\\
&&=\sum_{\mathbf x\in T^n(P)}P(\mathbf x).Q_{L,\mathbf x}^*\nonumber\\
&&+\sum_{\mathbf x\in\overline{T^n(P)}}P(\mathbf x).Q_{L,\mathbf x}^*.\label{eq:Sum1}
\end{eqnarray}
The second term in (\ref{eq:Sum1}) is upper-bounded by $P\left(\overline{T^n(P)}\right)$ which is less than $\frac{\epsilon}{2}$. We give an upper bound for the first term and show that for $L=2^{n(H^\prime(G,P)+\delta)}$ it tends to $0$ as $n\rightarrow\infty$. Now
\begin{eqnarray}
&&\sum_{\mathbf x\in T^n(P)}P(\mathbf x).Q_{L,\mathbf x}^*\leq\nonumber\\
&&P\left(T^n(P)\right).\max_{\mathbf x\in T^n(P)}Q_{L,\mathbf x}^*\leq\nonumber\\
&&\max_{\mathbf x\in T^n(P)}Q_{L,\mathbf x}^*.\nonumber
\end{eqnarray}
If an $n$-sequence $\mathbf x$ is not covered by an $L$-system, then $\mathbf x$ is not covered by any element of this system. Letting $\mathcal S_{\mathbf x}$ be the set of maximal independent sets covering the $n$-sequence $\mathbf x$, we have
\begin{equation}\label{eq:Bound13}
\max_{\mathbf x\in T^{n}(P)}Q_{L,\mathbf x}^*=\max_{\mathbf x\in T^n(P)}\left(1 - Q^*(\mathcal S_{\mathbf x})\right)^L.
\end{equation}
We obtain a lower bound for $Q^*(\mathcal S_{\mathbf x})$ by counting the $Q^*$-typical $n$-sequences of maximal independent sets covering $\mathbf x\in T^n(P)$. This number is greater than or equal to the $Q^*$-typical sequences $(\mathbf y,\mathbf B)$ with the first coordinate equal to $\mathbf x$. The equality of the first coordinate of the ordered pairs in $V\left(\Gamma\right)$ is an equivalence relation $p$ on the set $V\left(\Gamma\right)$.
Thus, using Lemma A. 3, the number of the $Q^*$-typical $n$-sequences of maximal independent sets is bounded from below by
\begin{equation}\label{eq:Bound14}
2^{nH(Q^*|q) - K_3\sqrt{n}},
\end{equation}
Let $K_4$ be a constant independent of $n$ and the density $a(Q^*)$. Then, applying Lemma A.2 to $\mathcal S$ and the marginal distribution $a(Q^*)$ of $Q^*$ over the maximal independent sets, we obtain the following lower bound on the probability $Q^*$ of the $a(Q^*)$-typical $n$-sequences of maximal independent sets,
\begin{equation}\label{eq:Bound15}
Q^*\geq 2^{-\left(nH(a(Q^*))+K_4\sqrt{n}\right)}.
\end{equation}
Combining (\ref{eq:Bound13}),(\ref{eq:Bound14}), and (\ref{eq:Bound15}), we get
\begin{eqnarray}
&&\max_{\mathbf x\in T^n(P)}Q_{L,\mathbf x}^*\leq\nonumber\\
&&\left(1 - \mathrm{exp_2}(-(nH(a(Q^*))+K_4\sqrt{n})+nH\left(Q^*|p\right) - K_3\sqrt{n})\right)^L.
\end{eqnarray}
Note that using (\ref{eq:mut}) we have
\[
H^\prime\left(G,P\right) = I(Q^*) = H\left(a(Q^*) - H(Q^*|p)\right),
\]
Therefore,
\[
\max_{\mathbf x\in T^n(P)}Q^*_{L,\mathbf x}\leq \left(1 - 2^{-(nH^\prime(G,P)+K_5\sqrt{n})}\right)^L,\qquad\text{for some constant $K_5$.}
\]
Then, using the inequality $\left(1 - x\right)^L\leq \mathrm{exp_2}(-Lx)$, the above inequality becomes
\begin{equation}\label{eq:Bound16}
\max_{\mathbf x\in T^n(P)}Q_{L,\mathbf x}^*\leq\mathrm{exp_2}\left(-L.2^{-(nH^\prime(G,P)+K_5\sqrt{n})}\right).
\end{equation}
Setting $L = 2^{(nH^\prime(G,P)+\delta)}$, (\ref{eq:Bound16}) becomes
\begin{eqnarray}\label{eq:Bound17}
\max_{\mathbf x\in T^n(P)}Q_{L,\mathbf x}^*&\leq&\mathrm{exp_2}\left(-(2^{nH^\prime(G,P)+\delta}-1).2^{-(nH^\prime(G,P)+K_5\sqrt{n})}\right)\nonumber\\
&\leq&\mathrm{exp_2}\left(-2^{n\delta-K_6\sqrt{n}}\right).
\end{eqnarray}
Substituting (\ref{eq:Sum1}) into (\ref{eq:Bound17}), we get
\[
\sum_{\mathbf x\in\mathcal X^n}P(\mathbf x).Q_{L,\mathbf x}^*\leq\mathrm{exp_2}\left(-2^{n\delta-K_6\sqrt{n}}\right)+\frac{\epsilon}{2},
\]
for $L = 2^{(nH^\prime(G,P)+\delta)}$. For sufficiently large $n$ the term $\mathrm{exp_2}\left(-2^{n\delta - K_6\sqrt{n}}\right)$ tends to zero, and (\ref{eq:Sum0}) implies
\[
\sum_{\mathbf S_1,\cdots,\mathbf S_L}Q_L^*\left(\mathbf S_1,\cdots,\mathbf S_L\right).P\left(U(\mathbf S_1,\cdots,\mathbf S_L)\right)\leq\epsilon.
\]
Thus, we conclude that for every $0<\epsilon<1$ and $\delta>0$, there exists a $\left(2^{n(H^\prime(G,P)+\delta)}\right)$-system covering a subgraph $G^{(n)}[U]$ of $G^{(n)}$ with probability of $U$ at least $1 - \epsilon$. Now, from Lemma A.1, the chromatic number of a graph is equal to the minimum number of maximal independent sets covering the graph. Therefore, for every $\delta>0$ there exists a subgraph $G^{(n)}[U]$ of $G^{(n)}$ with $U\in T_\epsilon^{(n)}$ such that
\[
\chi\left(G^{(n)}[U]\right)\leq 2^{n(H^\prime(G,P)+\delta)},
\]
Consequently,
\[
\min_{U\subset V(G^{(n)}), U\in T_\epsilon^n}\chi\left(G^{(n)}[U]\right)\leq 2^{n\left(H^\prime(G,P)+\delta\right)},\quad\text{for every}~\delta>0,
\]
Then, using the definition of $M(n,\epsilon)$, we get
\[
\frac{1}{n}\log_2M(n,\epsilon)\leq H^\prime\left(G,P\right)+\delta,\quad\text{for every}~\delta>0.
\]
And consequently, we get
\begin{equation}\label{eq:Bound18}
\limsup_{n\rightarrow\infty}\frac{1}{n}\log_2M\left(n,\epsilon\right)\leq H^\prime(G,P).
\end{equation}
Comparing (\ref{eq:Bound0}) and (\ref{eq:Bound18}), we obtain
\[
\lim_{n\rightarrow\infty}\frac{1}{n}\log_2M\left(n,\epsilon\right) = H^\prime\left(G,P\right),\quad\text{for every $\epsilon$ with $0<\epsilon<1$.}
\]
\qed

\newpage
\bibliographystyle{plain}
% This specifies the location of the file containing the bibliographic information.
% It assumes you're using BibTeX (if not, why not?).
\cleardoublepage % This is needed if the book class is used, to place the anchor in the correct page,
                 % because the bibliography will start on its own page.
                 % Use \clearpage instead if the document class uses the "oneside" argument
\phantomsection  % With hyperref package, enables hyperlinking from the table of contents to bibliography
% The following statement causes the title "References" to be used for the bibliography section:
\renewcommand*{\bibname}{References}

% Add the References to the Table of Contents
%\addcontentsline{toc}{chapter}{\textbf{References}}

\bibliography{uw-ethesis}
% Tip 5: You can create multiple .bib files to organize your references.
% Just list them all in the \bibliogaphy command, separated by commas (no spaces).

% The following statement causes the specified references to be added to the bibliography% even if they were not
% cited in the text. The asterisk is a wildcard that causes all entries in the bibliographic database to be included (optional).
\nocite{*}

\printindex
\end{document}